\documentclass[12pt]{article} \usepackage{fullpage}
\usepackage{amsmath,amssymb,amsthm} \usepackage[sort,numbers]{natbib}
\usepackage{mathrsfs} \usepackage{authblk}
\usepackage[ruled,vlined]{algorithm2e} \usepackage{algpseudocode}
\usepackage{amssymb, colordvi} \usepackage{multirow} \usepackage{graphicx}
\usepackage{blkarray} \usepackage{tikz} \usetikzlibrary{decorations}
\usepackage{centernot}
\pgfdeclaredecoration{multiple arrows}{draw}{
  \state{draw}[width=\pgfdecoratedinputsegmentlength]{ \draw [multiple arrows
  path/.try] (0,0) -- (\pgfdecoratedinputsegmentlength,0); } }
  \tikzset{multiple arrows/.style={multiple arrows path/.style={#1},
decoration=multiple arrows, decorate}}

\numberwithin{equation}{section} 
\newtheorem{theorem}{Theorem}
\numberwithin{theorem}{section}

\newtheorem{proposition}[theorem]{Proposition}
\newtheorem{example}[theorem]{Example} 
\newtheorem{lemma}[theorem]{Lemma}
\newtheorem{corollary}[theorem]{Corollary}
\newtheorem{definition}[theorem]{Definition}
\newtheorem{remark}[theorem]{Remark}

\newtheorem{question}[theorem]{Question}

\newtheorem{introtheorem}{Theorem}

\DeclareMathOperator{\Vertices}{Vertices} 
 
\newcommand{\C}{\mathbb{C}}
\newcommand{\Z}{\mathbb{Z}}

\DeclareMathOperator{\GL}{GL}

\DeclareMathOperator{\minor}{minor}

\DeclareMathOperator{\sign}{sign}

\DeclareMathOperator{\conv}{conv}

\newcommand{\R}{\mathbb{R}}
\newcommand{\Q}{\mathbb{Q}}

\DeclareMathOperator{\Vol}{Vol}

\newenvironment{myenumerate}{ \vspace{-3pt} \begin{enumerate}
  \setlength{\itemsep}{0pt} \setlength{\parskip}{1pt}
    \setlength{\parsep}{-8pt}} {\end{enumerate} \vspace{-3pt} }

\begin{document}

\title{
A Polyhedral Method for Sparse Systems with many Positive Solutions
} 

\author[1]{Fr\'ed\'eric Bihan, 
Francisco Santos\thanks{Francisco Santos is partially supported by grant MTM2014-54207-P  of the Spanish
  Ministry of Science, by the Einstein Foundation Berlin and, while he was in residence at the Mathematical Sciences Research Institute in Berkeley, California during the Fall 2017 semester, by the Clay Institute and the National Science Foundation (Grant No. DMS-1440140).}, and 
  Pierre-Jean Spaenlehauer}

\date{}

\maketitle

\begin{abstract}
We investigate a version of Viro's method for constructing polynomial systems
with many positive solutions, based on regular triangulations of the Newton
polytope of the system.  The number of positive solutions obtained with our
method is governed by the size of the largest \emph{positively decorable}
subcomplex of the triangulation. Here, positive decorability is a property that
we introduce and which is dual to being a subcomplex of some regular
triangulation. 
Using this duality, we produce large positively decorable subcomplexes of the
boundary complexes of cyclic polytopes.
As a byproduct we get new lower bounds, some of them being the best currently
known, for the maximal number of positive solutions of polynomial systems with
prescribed numbers of monomials and variables.  We also study the asymptotics of
these numbers and observe a log-concavity property.
\end{abstract}

\section{Introduction}
Positive solutions of multivariate polynomial systems are central objects in
many applications of mathematics, as they often contain meaningful information,
e.g. in robotics, optimization, algebraic statistics, study of multistationarity in
chemical reaction networks, etc.  In the 70s,
foundational results by Kushnirenko~\cite{Kouch}, Khovanskii~\cite{Khov} and
Bernstein~\cite{Bern} laid the theoretical ground for the study of the
algebraic structure of polynomial systems with prescribed conditions on the set
of monomials appearing with nonzero coefficients.  As a particular case of more
general bounds, Khovanskii \cite{Khov} obtained an upper bound on the number of
non-degenerate positive solutions which depends only on the dimension of the
problem and on the number of monomials. 

More precisely, our main object of interest in this paper is the function
$\Xi_{d,k}$ defined as the maximal possible number of non-degenerate
solutions in $\R_{>0}^d$ of a polynomial system $f_1=\dots=
f_d=0$, where $f_1,\ldots, f_d\in\mathbb R[X_1,\ldots, X_d]$
involve at most $d+k+1$ monomials with nonzero coefficients. Here,
\emph{non-degenerate} means that the Jacobian matrix
of the system is invertible at the
solution.
Finding sharp bounds for $\Xi_{d,k}$ is a notably hard problem, see~\cite{Sbook}.
The current knowledge can be briefly summarized as follows (see~\cite{BRS,BS}): 
$$\forall
d,k>0,\quad \max((\lfloor k/d\rfloor+1)^d, (\lfloor
d/k\rfloor+1)^k)\le \Xi_{d,k}\le
(e^2+3)2^{\binom{k}2}d^k/4.$$
Another important and recent lower bound is $\Xi_{2, 2}\geq 7$~\cite{Boulos}.

In this paper we introduce a new technique to construct fewnomial systems with
many positive roots, based on the notion of \emph{positively decorable}
subcomplexes in a regular triangulation of the point configuration given by the exponent
vectors of the monomials. Using this method we obtain new lower bounds for
$\Xi_{d,k}$. Combining it with a log-concavity property, we obtain systems which admit
asymptotically more positive solutions than previous constructions for a large
range of parameters.

\smallskip

{\bf Main results.} Consider a regular full-dimensional pure simplicial complex
$\Gamma$ supported on a point configuration  $\mathcal A=\{{w_1,\ldots, w_n}\}
\subset\Z^d$, by which we mean that $\Gamma$ is a pure $d$-dimensional subcomplex of a 
regular triangulation of $\mathcal A$ (see~Definition~\ref{def:regular} and
Proposition~\ref{prop:regular}). Consider also a map $\phi: \mathcal
A\rightarrow \R^d$.
This map will be used to construct a polynomial system where the coefficients
of the monomial $w_i$ will be obtained from $\phi(w_i)$.
We call a facet $\tau=\conv(w_{i_1},\ldots,w_{i_{d+1}})$ of $\Gamma$
\emph{positively decorated} by $\phi$ if $\phi(\{w_{i_1},\ldots,w_{i_{d+1}}\})$
positively spans $\R^d$. 
We are interested in sparse polynomial systems
\begin{equation}
\label{E:system}
f_{1}(X_1,\ldots, X_d)=\cdots=f_{d}(X_1,\ldots, X_d)=0
\end{equation}
with real coefficients and support contained in $\mathcal A$: this means that all exponent vectors $w\in\Z^d$ of the monomials $X^{w}$ appearing with a nonzero coefficient in at least one equation are in $\mathcal A$. Our starting point is the following result:

\begin{introtheorem}[Theorem~\ref{thm:nbpossols}]
\label{introthm:nbpossols}
There is a choice of coefficients --- which can be constructed from the map
$\phi$ ---
which produces a sparse system supported on $\mathcal A$ such that
the number
of non-degenerate positive solutions of \eqref{E:system} 
is bounded below by the number of facets in $\Gamma$ which are
positively decorated by $\phi$.
\end{introtheorem}

This theorem is a version of Viro's method which was used by Sturmfels~\cite{St} to construct sparse polynomial systems all solutions of which are real.
Viro's method (\cite{Vir}, see also \cite{R, St3,B3}) is one of the
roots of tropical geometry and it has been used for
constructing real algebraic varieties with interesting topological and combinatorial properties. 

We then apply this theorem to the problem of constructing fewnomial systems with many
positive solutions. For this we construct large simplicial complexes 
that are regular and \emph{positively decorable} (that is, all their
facets can be positively decorated with a certain $\phi$), obtained as subcomplexes of the boundary of cyclic polytopes.
Combinatorial techniques allow us to count the  simplices of these complexes, which
gives us new explicit lower bounds on $\Xi_{d,k}$. More precisely, for all $i,j\in\mathbb Z_{>0}$, set
\[
F_{i,j}=D_{i,j} + D_{i-1,j-1}
\]
where $D_{i,j}$ is the $(i,j)$-th \emph{Delannoy number}~\cite{BS05}, defined as
\begin{equation}
\label{eq:Delannoy}
D_{i,j}:=\sum_{\ell=0}^{\min\{i,j\}}\frac{(i+j-\ell)!}{(i-\ell)!(j-\ell)!\ell!}
=\sum_{\ell=0}^{\min\{i,j\}} 2^\ell \binom i\ell \binom j\ell.
\end{equation}

\begin{introtheorem}[Corollary~\ref{coro:S_decorable}, Remark~\ref{rem:S_decorable}]
\label{introthm:S_decorable}
  For every $i,j\in \Z_{>0}$ we have 
  \[
  \Xi_{2i-1,2j} \geq  F_{i,j}, \quad
  \Xi_{2i-1,2j-1} \geq \frac{j}{i+j} \,  F_{i,j}, \quad
  \Xi_{2i,2j} \geq \frac{i+1}{i+j+1} \,   F_{i+1,j}, \quad
  \Xi_{2i,2j-1} \geq  2 \, F_{i,j-1}.
  \]
\end{introtheorem}

We are then interested in the asymptotics of $\Xi_{d,k}$ for big $d$ and $k$.
One way to make sense of this is the following.

\begin{introtheorem}[Theorem~\ref{thm:limit_exists}]
\label{introthm:limit_exists}
For all $k,d \in \Z_{>0}$ the limit
$\xi_{d,k}:=\lim_{n\rightarrow\infty}(\Xi_{dn, kn})^{1/(dn+kn)} \in [1,\infty]$
exists. Moreover, this limit  depends only on the ratio $d/k$ and it is bounded from below by ${\Xi_{d, k}}^{1/(d+k)}$.
 \end{introtheorem}
 
 Analyzing the asymptotics of Delannoy numbers
 leads to the following new lower bound, which also depends only on $d/k$:
 
 \begin{introtheorem}[Theorem~\ref{thm:lower_bound}, Corollary~\ref{coro:lower_bound_xi}]
\label{introthm:lower_bound}
 For all $k,d\in \Z_{>0}$ we have
 \[
 \xi_{d,k}
 \geq
 \left( \frac{\sqrt{d^2+k^2}+k}d\right)^{\frac{d}{2(d+k)}}
 \left( \frac{\sqrt{d^2+k^2}+d}k\right)^{\frac{k}{2(d+k)}}.
 \]
\end{introtheorem}

This statement allows us
to improve the lower bounds on $\xi_{d,k}$ for $0.2434 < d /(d+k) <0.3659$ and
for $0.6342 < d/(d+k) <0.7565$, see Figure~\ref{fig:graph_bound}.
In fact, Theorem~\ref{thm:limit_exists} implies that the limit $\xi_{\alpha,\beta}:=
\lim_{n\rightarrow\infty}(\Xi_{\alpha n,\beta n})^{1/(\alpha n+\beta n)}$ exists for any positive rational numbers $\alpha,\beta$.
It is convenient to look at $\xi$ along the segment $\alpha +\beta=1$. This is
no loss of generality since $\xi_{d,k}=\xi_{\alpha,1-\alpha}$ for
$\alpha=d/(d+k)$ and it has the nice property that
the function $\xi: \alpha \mapsto \xi_{\alpha,1-\alpha}$, $\alpha \in (0,1) \cap \Q$, is
log-concave (Proposition~\ref{prop:concavity}).
Therefore, convex hulls of lower bounds for $\log \xi_{\alpha,1-\alpha}$ also produce lower bounds for this function.
With this observation, the methods in this paper improve the previously known lower bounds for $\xi_{d,k}$
for all $d,k$ with $d/(d+k) \in (0.2,0.5)\cup (0.5, 0.8)$, see Figure~\ref{fig:graph_bound_log}.

Our bounds also raise some important questions about $\xi_{d,k}$. Notice that
log-concavity implies that if $\xi$ is infinite somewhere then it is infinite
everywhere (Corollary~\ref{coro:concavity}). In light of this, we pose the following
question:
\begin{question}Is $\xi_{d,k}$ finite for some (equivalently, for every) $d,k>0$? 
That is to say, is there a global constant $c$ such that $\Xi_{d,k} \le c^{k+d}$ for all $k,d \in \Z_{>0}$?
\end{question}
In fact, we do not know whether $\Xi_{d,k}$ admits a singly exponential upper bound, since the best known general upper bound (see Eq.~\eqref{eq:upperbound}) is only of type $2^{O(k^2+k\log d)}$. 
Compare this to Problem 2.8 in Sturmfels~\cite{St} (still open), which asks whether $\Xi_{d,k}$ is polynomial for fixed $d$. In a sense, Sturmfels' formulation is related to the behaviour of $\xi_{\alpha,\beta}$ when $\alpha/\beta \approx 0$, although the answer  to it  might be positive even if $\xi$ is infinite. (Think, e.g., of $\Xi_{d,k}$ growing as $\min\{d^k,k^d\}$). Our formulation looks at $\Xi$ globally and gives the same role to $d$ and $k$, which is consistent with Proposition~\ref{prop:superadditive}.

  Another intriguing question is whether $\xi_{d,k}=\xi_{k,d}$ or, more
  strongly, whether $\Xi_{d,k}=\Xi_{k,d}$.
This symmetry between $d$ and $k$ holds true for all known lower bounds and
exact values, including the lower bounds for $\xi_{d,k}$ obtained with our construction where the symmetry is a consequence of a Gale-type duality between
regular and positively decorable complexes (see Corollary~\ref{coro:galeposdecreg} and Theorem~\ref{thm:boundXi}). This duality is also instrumental in our proof that the complexes used for Theorem~\ref{introthm:S_decorable} are positively decorable.

Although not needed for the rest of the paper, we also show that positive decorability is related to two classical properties in topological combinatorics:
\begin{introtheorem}[Theorem~\ref{thm:balanced_vs_bipartite}]
\label{introthm:balanced_vs_bipartite}
For every pure orientable simplicial complex one has
\[
\text{balanced}\,\Longrightarrow\,
\text{positively decorable}\,\Longrightarrow\,
\text{bipartite}.
\]
\end{introtheorem}
Under certain hypotheses (e.g., for complexes that are simply connected manifolds with or without boundary) the reverse implications also hold
(Corollary~\ref{coro:balanced_vs_bipartite_2}).

{\bf Organization of the paper.} In Section~\ref{sec:prelim} we present
classical bounds for $\Xi_{d,k}$, introduce the quantity $\xi_{d,k}$, and prove
Theorem~\ref{introthm:limit_exists}, plus the log-concavity property.  Section
\ref{sec:oriented} describes the Viro's construction used throughout the paper
and proves Theorem~\ref{introthm:nbpossols}.  In Section \ref{sec:dualgraph},
we show the duality between positively decorable and regular 
complexes, and Section~\ref{sec:bipartite} relates positive decorability to balancedness and bipartiteness.
Section \ref{sec:cyclic} contains our main construction, based on
cyclic polytopes, and shows the lower bounds stated in
Theorem~\ref{introthm:S_decorable}. This bound is analyzed and compared to
previous ones in Section~\ref{sec:asymptotics}, where we prove
Theorem~\ref{introthm:lower_bound}. In Section~\ref{sec:Xi_vs_R}, we
investigate the potential of the proposed method and we show that the number of
positive solutions that can be produced by this method is inherently limited by
the upper bound theorem for polytopes.

\paragraph{Acknowledgements.} We are grateful to Alin Bostan, Alicia Dickenstein, Louis Dumont, \'Eric Schost, Frank Sottile and Bernd Sturmfels for helpful discussions.

\section{Preliminaries on $\Xi_{d,k}$}
\label{sec:prelim}
Here we review what is known about the function $\Xi_{d,k}$, defined as the maximum possible number of positive non-degenerate solutions of $d$-dimensional systems with $d+k+1$ monomials. 
Finiteness of  $\Xi_{d,k}$ follows from the work of Khovanskii \cite{Khov}.
The currently best known general upper bound for $\Xi_{d,k}$ for arbitrary $d$ and $k$ is proved by
Bihan and Sottile~\cite{BS}:
\begin{equation}
\label{eq:upperbound}
\Xi_{d,k}\leq \dfrac{e^2+3}{4} 2^{\binom{k}2}d^k, \qquad \forall k,d \in \Z_{>0}.
\end{equation}

The following proposition summarizes what is known about 
 lower bounds of $\Xi_{d,k}$:

\begin{proposition}
\label{prop:superadditive}
\begin{enumerate}
  \item $\Xi_{d+d',k+k'} \ge \Xi_{d,k} \, \Xi_{d',k'}$ for all $d,d',k,k'\in\mathbb Z_{>0}$.
  \item $\Xi_{1,k}=k+1$ for all $k\in\mathbb Z_{>0}$ (Descartes).
  \item $\Xi_{d,1}=d+1$ for all $d\in\mathbb Z_{>0}$ (Bihan~\cite{B2}). 
  \item $\Xi_{2,2}\ge 7$ (B. El Hilany \cite{Boulos}).
\end{enumerate}
\end{proposition}

\begin{proof}
Let $A\subset \Z^d$ and $A'\subset\Z^{d'}$ be supports of systems in $d$ and
$d'$ variables with $d+k+1$ and $d'+k'+1$ monomials achieving the bounds
$\Xi_{d,k}$ and $\Xi_{d',k'}$.  
Without loss of generality, we can assume that both
$A$ and $A'$ contain the origin; Indeed, translating the supports amounts to
multiplying the whole system by a monomial, which does not affect the number of
positive roots. Then $(A\times\{0\})\cup(\{0\}\times A')\subset \Z^{d+d'}$ has
$(d+d')+(k+k')+1$ points and supports a system (the union of the original
systems) with $\Xi_{d,k} \Xi_{d',k'}$ nondegenerate positive solutions (the
Cartesian product of the solutions sets of the original systems).  Therefore,
$\Xi_{d+d',k+k'} \ge \Xi_{d,k} \Xi_{d',k'}$. 

The equality $\Xi_{1,k}=k+1$ comes from the fact that a
univariate polynomial with $k+2$ monomials cannot have more than $k+1$ positive
solutions by Descartes' rule of signs (and the polynomial $\prod_{i=1}^{k+1}(x-i)$ reaches this bound). 

Finally, $\Xi_{d,1}=d+1$ was proved in \cite[Thm. A]{B2} and $\Xi_{2,2} \geq 7$  has been recently shown by B. El Hilany \cite[Thm.~1.2]{Boulos} using tropical geometry. 
\end{proof}

\begin{remark}\label{increasing}
It is known that $\Xi_{d,0}=1$ (see Proposition \ref{prop:posorthantsimplex}). Moreover, $\Xi_{d,k+1} \geq \Xi_{d,k}$ is obvious (adding one monomial with a very small coefficient
does not decrease the number of nondegenerate positive solutions). Then,
by setting $\Xi_{0,k}=1$, Part (1) of Proposition~\ref{prop:superadditive} can be
extended to allow zero values for $d$ and $k$.
Consequently, $\Xi_{d',k'} \geq \Xi_{d,k}$ if $d' \geq d$ and $k' \geq k$.
\end{remark}

The following consequences of Proposition~\ref{prop:superadditive} have been
observed before. Part (1) comes from a system of univariate polynomials in
independent variables, and part (2) was proved by Bihan, Rojas and Sottile in \cite{BRS}.

\begin{corollary}
\label{coro:k,d=1}
\begin{enumerate}
\item If $k_1+\dots+k_d=k$ is an integer partition of $k$, then we have $\Xi_{d,k} \geq \prod_{1\leq i\leq d}(k_i+1)$.
In particular,
\begin{equation}\label{E:sec}
\Xi_{d,k} \geq 
(\lfloor k/d\rfloor+1)^d.
\end{equation}
\item If $d_1+\dots+d_k=d$ is an integer partition of $d$, then $\Xi_{d,k} \geq \prod_{1\leq i\leq d}(d_i+1)$.
In particular
\begin{equation}\label{E:first}
\Xi_{d,k} \geq 
(\lfloor d/k\rfloor+1)^k.
\end{equation}
\end{enumerate}
\qed
\end{corollary}

Observe that both bounds specialize to 
\begin{equation}\label{eq:Xidd}\Xi_{d,d}\ge 2^d\end{equation} 
for $k=d$, but a better bound of $\Xi_{2d,2d}\ge 7^d$ follows from parts 1 and 4 of Proposition~\ref{prop:superadditive}.

In Section~\ref{sec:asymptotics} we will be interested in the asymptotics of
$\Xi_{d,k}$ for big $d$ and $k$.

\begin{theorem}
\label{thm:limit_exists}
Let $d,k\in \Z_{>0}$. Then the following limit exists:
\[
\lim_{n\to\infty} (\Xi_{dn,kn})^{1/({dn+kn})} \in [1,\infty].
\]
Moreover the limit depends only on the ratio $d/k$ and it is bounded from below by $(\Xi_{d,k})^{1/(d+k)}$.
\end{theorem}

\begin{proof}
For each $n$, let $a_n:= \log (\Xi_{dn,kn})$, so that the limit we want to compute is $\lim_{n\to\infty} e^{a_n/(d+k)n}$ and we can instead look at 
$\lim_{n\to\infty} (a_n/(d+k)n)$. 
Since $(a_n)_{n\in \Z_{> 0}}$ is increasing (Remark \ref{increasing}) and $a_{pn_0} \ge p a_{n_0}$ for every
positive integer $p$ (Proposition~\ref{prop:superadditive}), we have $a_n \ge
  \lfloor\frac{n}{n_0}\rfloor a_{n_0}$ for all $n, n_0\in\Z_{> 0}$.
Thus:
\[
\liminf_{n\to \infty} \frac{a_n}{n} \ge \liminf_{n\to \infty}
\left\lfloor\frac{n}{n_0}\right\rfloor \frac{a_{n_0}}{n} = 
\frac{a_{n_0}}{n_0} , \quad \forall n_0\in \Z_{>0}.
\]
Consequently,
\[
\liminf_{n\to\infty} \frac{a_n}{(d+k)n} =\frac1{d+k}\liminf_{n\to\infty}
\frac{a_n}n  \ge \frac1{d+k} \, \frac{a_{n_0}}{n_0}, \qquad \forall n_0 \in
\Z_{>0}.
\]
In particular,
\[
  \liminf_{n\to\infty} \frac{a_n}{(d+k)n} 
  \ge 
  \sup_{n\in\Z_{> 0}} \frac{a_n}{(d+k)n}
  \ge
   \limsup_{n\to\infty} \frac{a_n}{(d+k)n},
\]
which implies that the limit exists and equals the supremum. To show that the
limit depends only on the ratio $d/k$, observe that if $(d,k)$ and $(d',k')$ are
proportional vectors then the sequences $(a_{n}/(d+k)n)_{n\in \Z_{> 0}}$ and
$(a'_{n}/(d'+k')n)_{n\in \Z_{> 0}}$ (where $a'_n:= \log (\Xi_{d'n,k'n})$) have a common subsequence.
\end{proof}

Note that the statement implies the existence of the limit
\[
\xi_{\alpha,\beta} =
\lim_{\substack{n\rightarrow\infty\\ \alpha n,\beta n\in \Z}}(\Xi_{ \alpha n, \beta n})^{1/(\alpha n + \beta n)} \in [1,\infty]
\]
for any positive rational numbers $\alpha, \beta\in \Q_{>0}$ and for $\alpha=\frac{d}{d+k}$ (where $d,k \in \Z_{>0}$)
we have $\xi_{\alpha,1-\alpha}=\lim_{n\to\infty} (\Xi_{dn,kn})^{1/({dn+kn})}$.
Also, since the limit in Theorem~\ref{thm:limit_exists} depends only on $d/k$, we only need to consider the function $\xi$ for one point along each ray in the positive orthant. We choose the segment defined by $\alpha+\beta=1$ because along this segment $\xi$ is log-concave:

\begin{proposition}\label{prop:concavity}
The function 
$\alpha \mapsto \xi_{\alpha,1-\alpha}$ is log-concave over $(0,1)\cap \Q$.
\end{proposition}

\begin{proof}
For any integer $n$ and any $(\alpha,\beta) \in \Q_{>0}^2$ with $\alpha n, \beta n\in \Z$, let $a_n(\alpha,\beta)=\log (\Xi_{\alpha n, \beta n})$.
The statement is that
for any $(\alpha,\beta)$, $(\alpha',\beta')$ in $\Q_{>0}^2$ and any $\theta \in
[0,1]\cap \Q$, we have
\begin{equation}\label{concave}
\lim_{n\to\infty} \frac{1}{n} a_n(\theta (\alpha,\beta) +(1-\theta) (\alpha',\beta')) \geq \theta \lim_{n\to\infty} \frac{1}{n}a_n(\alpha,\beta)+
(1-\theta) \lim_{n\to\infty} \frac{1}{n}a_n(\alpha',\beta').
\end{equation}
Here and in what follows only values of $n$ where $\alpha n,\theta\alpha n,$ etc.~are integers are considered. This is enough since they form an infinite sequence and the limit $\xi$ is independent of the subsequence considered.

Using Proposition~\ref{prop:superadditive}, together with Remark~\ref{increasing}, we get
\begin{eqnarray*}
a_n(\theta (\alpha,\beta) +(1-\theta) (\alpha',\beta')) 
&=& \log (\Xi_{ \theta \alpha n + (1-\theta) \alpha'n, \theta \beta n +(1-\theta)\beta'n}) \\
&\geq& \log (\Xi_{ \theta \alpha n, \theta \beta n})+
\log (\Xi_{ (1-\theta) \alpha' n, (1-\theta) \beta' n}) \\
&=& a_n(\theta \alpha, \theta \beta)+ a_n((1-\theta) \alpha', (1-\theta) \beta').
\end{eqnarray*}
It remains to note that $\lim_{n\to\infty}\frac{1}{n}  a_n(\theta \alpha, \theta \beta)=\theta \lim_{n\to\infty} \frac{1}{n} a_n(\alpha,\beta)$ for any $(\alpha,\beta)$ in $\Q_{>0}^2$ and any $\theta \in
[0,1]\cap \Q$.
\end{proof}

One interesting consequence of log-concavity is:

\begin{corollary}
\label{coro:concavity}
The function $\xi_{\alpha,\beta}$ is either finite for all $(\alpha,\beta) \in \Q_{>0}^2$ or infinite for all $(\alpha,\beta) \in \Q_{>0}^2$.
\end{corollary}

\begin{proof}
Since $\xi_{\alpha,\beta}$ depends only on $\alpha / \beta$, there is no loss
of generality in assuming $\beta=1-\alpha$ and $\alpha\in (0,1)\cap \mathbb Q$.
Suppose $\xi_{\alpha,1-\alpha}=\infty$ for some $\alpha\in (0,1)\cap \Q$ and let us show that $\xi_{\beta,1-\beta}=\infty$ for every other $\beta\in (0,1)\cap \Q$.  For this, let $\gamma=(1+\epsilon)\beta -\epsilon \alpha$ for a sufficiently small $\epsilon\in \Q_{>0}$, so that $\gamma \in (0,1)\cap \Q$. Then $\beta = \frac1{1+\epsilon} \gamma + \frac\epsilon{1+\epsilon} \alpha$. By log-concavity:
\[
\xi_{\beta,1-\beta} \ge {\xi_{\gamma,1-\gamma}}^{\frac1{1+\epsilon}} \ {\xi_{\alpha,1-\alpha}}^{\frac\epsilon{1+\epsilon}} = \infty.
\]
\end{proof}

\begin{remark}
\label{rem:rational}
Although we have defined $\xi$ only for rational values in order to avoid technicalities, log-concavity and Proposition~\ref{prop:superadditive} easily imply that $\xi$ admits a unique continuous extension to $\alpha,\beta\in \R_{>0}$ and that
this extension satisfies
\[
\xi_{\alpha,\beta}=
\lim_{n\rightarrow\infty}(\Xi_{\lfloor \alpha n\rfloor,\lfloor \beta n\rfloor})^{1/(\alpha n+\beta n)} =
\lim_{n\rightarrow\infty}(\Xi_{\lceil \alpha n\rceil,\lceil \beta n\rceil})^{1/(\alpha n+\beta n)}.
\]
\end{remark}

\section{Positively decorated simplices and Viro polynomial systems}
\label{sec:oriented}

We start by considering systems of $d$ equations in $d$ variables whose support
$\mathcal A=\{w_1,\ldots, w_{d+1}\}\subset \Z^d$ is the set of vertices of a $d$-simplex.
This case is a basic building block in our construction.

\begin{definition} A $d\times (d+1)$ matrix $M$ with real entries is called
  \emph{positively spanning} if all the values $(-1)^i{\minor}(M,i)$ are nonzero and have 
  the same sign, where $\minor(M,i)$ is the determinant of the square matrix
  obtained by removing the $i$-th column.  \end{definition}

The terminology ``positively spanning'' comes from the fact that if $\mathcal A=\{w_1,\ldots, w_{d+1}\}$
is the set of columns of $M$, saying that $M$ is positively spanning is equivalent to
saying that any vector in $\R^d$ is a linear combination \emph{with positive
coefficients} of $w_1,\ldots, w_{d+1}$.

\begin{proposition} \label{P:invariant} Let $M$ be a full rank $d\times (d+1)$
  matrix with real coefficients. The following statements are equivalent:
  \begin{myenumerate} \item the matrix $M$ is positively spanning; 
    \item for any
      $L\in\GL_d(\R)$, $L\cdot M$ is a positively spanning matrix; 
      \item for any
      permutation matrix $P\in\mathfrak S_{d+1}$, $M\cdot P$ is a positively
      spanning matrix; 
    \item all the coordinates of any non-zero vector in the kernel of
      the matrix are non-zero and have the same sign;
      \item the origin belongs to the interior of the convex hull of the column vectors of $M$.
      \item every vector in $\R^d$ is a nonnegative linear combination of the columns of $M$.
 \item there is no $w\in \R^d\setminus \{0\}$ s.t. $w\cdot M\ge 0$.
  \end{myenumerate} 
\end{proposition}

\begin{proof} The equivalence $(1)\Leftrightarrow (4)$ follows from Cramer's
rule, while $(2)\Rightarrow (1)$ and $(3)\Rightarrow (1)$ are proved directly by
  instantiating $L$ and $P$ to the identity matrix.  The implication $(1)\Rightarrow (2)$
  follows from 
  \[
  \sign((-1)^i{\minor}(L\cdot M,i)) = \sign(\det(L))\cdot
  \sign((-1)^i{\minor}(M,i)),
  \] 
  while $(3) \Leftrightarrow (4)$ is a consequence of
  the fact that permuting the columns of $M$ is equivalent to permuting the
  coordinates of the kernel vectors.
  The equivalence between $(4)$ and
  $(5)$ follows from the definition of convex hull:  
  the origin is in the
  interior of the convex hull of the column vectors if and only if it can be
written as a positive linear combination of these vectors. The equivalence between $(5)$ and
  $(6)$ is obvious and the equivalence between $(5)$ and $(7)$ follows from Farkas' Lemma.
   \end{proof}

\begin{proposition}\label{prop:posorthantsimplex}
Assume that $\mathcal A=\{w_1,\ldots, w_{d+1}\}$ is the set of vertices of a
$d$-simplex in $\R^d$, and consider the polynomial system with real coefficients
\[
f_i(X)=\sum_{j=1}^{d+1}C_{ij} X^{w_j},\quad 1\leq i\leq d.
\]
The system $f_1(X)=\dots=f_d(X)=0$ has at most one non-degenerate positive solution
and it has one non-degenerate positive solution
if and only if the $d\times (d+1)$ matrix $C=(C_{ij})$ is positively spanning. 
\end{proposition}

\begin{proof}
Multiplying the system by $X^{-w_{d+1}}$ (which does not change
the set of positive solutions), we can assume without loss of
generality that $w_{d+1}=\mathbf 0$.
Consider the monomial map  $(X_1,\dots,X_d) \to (X^{w_1}, \dots, X^{w_d})$ which bijects the positive orthant to itself. The map is
invertible since $\mathcal A$ is affinely independent, and the
inverse map transforms the system $f_1(X)=\dots=f_d(X)=0$ into a linear system with $C$ as its coefficient matrix.
Then the statement follows from Proposition \ref{P:invariant}: by part (4) of the proposition, the unique solution of the linear system lies in the positive orthant if, and only if, $C$ is positively spanning.
\end{proof}

Consider now a set $\mathcal A=\{w_1,\ldots, w_{n}\}\subset \Z^d$ and assume
that its convex hull is a full-dimensional polytope $Q$. Let $\Gamma$ be a
triangulation of $Q$ with vertices in $\mathcal A$. Assume that $\Gamma$ is a
\emph{regular} triangulation, which means that there exists a convex function
$\nu:Q \rightarrow \R$ which is affine on each simplex of $\Gamma$, but not
affine on the union of two different facets of $\Gamma$ (such triangulations
are sometimes called coherent or convex in the literature; see \cite{LRS} for
extensive information on regular triangulations).  We say that $\nu$, which is
sometimes called the \emph{lifting} function, \emph{certifies} the regularity
of $\Gamma$.  Let $C$ be a $d\times n$ matrix with real entries.
This matrix defines a map $\phi:\mathcal A\rightarrow \mathbb R^d$ as in Theorem~\ref{introthm:nbpossols},
by setting $\phi(w_i)$ to the $i$th column of $C$.
We say that
$C$ \emph{positively decorates} a facet $\tau={\rm
conv}(w_{i_1},\ldots,w_{i_{d+1}})\in\Gamma$ if the $d\times(d+1)$ submatrix of
$C$ given by the columns numbered by $\{i_1,\ldots, i_{d+1}\}$ is positively
spanning.
The associated \emph{Viro polynomial system} is
\begin{equation}
\label{E:Virosystem}
f_{1,t}(X)=\cdots=f_{d,t}(X)=0,
\end{equation}
where $t$ is a positive parameter and
\[
f_{i,t}(X)=\sum_{j=1}^n C_{ij}t^{\nu(w_j)}X^{w_j} \in
\R[X_1,\ldots,X_d] , \quad i=1,\ldots,d.
\]

The following result is a variation of the main theorem in~\cite{St}. There,
the number of \emph{real} roots of the system \eqref{E:Virosystem} is bounded
below by the number of \emph{odd} facets in $\Gamma$
(facets with odd normalized volume).
Proposition~\ref{prop:posorthantsimplex} allows us to change that to a lower bound for \emph{positive} roots in terms of \emph{positively decorated} simplices.

\begin{theorem}\label{thm:nbpossols}
Let $\Gamma$ be a regular triangulation of $\mathcal A=\{w_1,\ldots, w_{n}\}\subset \Z^d$ and let $C\in \R^{d\times n}$. Then,
there exists $t_0\in\R_+$ such that for all $0<t<t_0$ the number
of non-degenerate positive solutions of the system \eqref{E:Virosystem}
is bounded from below by the number of facets in $\Gamma$ which are
positively decorated by $C$.
\end{theorem}

\begin{proof}Let $\tau_1,\ldots,\tau_m$ be the facets of $\Gamma$ which are
  positively decorated by $C$. For all $\ell \in \{1,\ldots, m\}$, the function $\nu$ is affine on $\tau_{\ell}$, thus
there exist $\alpha_{\ell}=(\alpha_{1 \ell},\ldots,\alpha_{d \ell}) \in \R^d$ and $\beta_{\ell} \in \R$ such that
$\nu(x)=\langle \alpha_{\ell},x \rangle +\beta_{\ell}$ for any
$x=(x_1,\ldots,x_d)$ in the simplex $\tau_{\ell}$.
Set $Xt^{-\alpha_{\ell}}=(X_1t^{-\alpha_{1 \ell}},\ldots, X_dt^{-\alpha_{d
\ell}})$. Since $\nu$ is convex and not affine on the union of two distinct
facets of $\Gamma$,
we get
\begin{equation}\label{E:Virosyst}
\frac{f_{i,t}(Xt^{-\alpha_{\ell}})}{t^{\beta_{\ell}}}
=
f_{i}^{(\ell)}(X)+r_{i,t}^{(\ell)}(X), \quad i=1,\ldots,d,
\end{equation}
where $f_{i}^{(\ell)}(X)=\sum_{w_j \in \tau_{\ell}} C_{ij}X^{w_j}$ and
$r_{i,t}^{(\ell)}(X)$ is a polynomial each of whose coefficients is equal to a positive power of $t$ multiplied by a coefficient of $C$.
Since $\tau_{\ell}$ is positively decorated by $C$, the system $f_{1}^{(\ell)}(X)=\cdots=f_{d}^{(\ell)}(X)=0$ has
one non-degenerate positive solution $z_{\ell}$ by Proposition \ref{prop:posorthantsimplex}. It follows that the system $f_{1}^{(\ell)}(X)+r_{1,t}^{(\ell)}(X)=\cdots=f_{d}^{(\ell)}(X)+r_{d,t}^{(\ell)}(X)=0$ has a non-degenerate solution close to $z_{\ell}$ for $t>0$ small enough. More precisely, for all $\varepsilon>0$, there exists
  $t_{\varepsilon, \ell}>0$ such that for all
  $0<t<t_{\varepsilon, \ell}$, there exists a non-degenerate solution $z_{\ell, t}$ of $f_{1}^{(\ell)}(X)+r_{1,t}^{(\ell)}(X)=\cdots=f_{d}^{(\ell)}(X)+r_{d,t}^{(\ell)}(X)=0$ such that $\lVert z_{\ell, t} -
  z_{\ell}\rVert <\varepsilon$. Then using \eqref{E:Virosyst} we get $f_{1,t}(z_{\ell, t}t^{-\alpha_{\ell}})=\cdots= f_{d,t}(z_{\ell, t}t^{-\alpha_{\ell}})=0$.
Now, choose $\varepsilon$ small enough so that the balls of radius $\varepsilon$ centered at $z_1,\ldots,z_m$ are contained in a compact set $K \subset \R_{>0}^d$. Since the vectors
$\alpha_\ell$ are distinct, there exists $\tau>0$
such that for all $0<t<\tau$ the sets $K \cdot t^{-\alpha_{\ell}}= \{(X_1t^{-\alpha_{1
\ell}},\ldots, X_dt^{-\alpha_{d \ell}}) \, | \, (X_1,\ldots,X_d) \in K\}$,
$\ell=1,\ldots,m$, are pairwise disjoint. Set $t_0 = \min(\tau, t_{\varepsilon,
1}, \ldots, t_{\varepsilon, m})$. Then, for $0<t<t_0$ each of these sets
$K \cdot t^{-\alpha_{\ell}}$ contains a non-degenerate positive solution $z_{\ell, t}t^{-\alpha_{\ell}}$
of the system \eqref{E:Virosystem}.
\end{proof}

\section{Duality between regular and positively decorable complexes}
\label{sec:dualgraph}

In this section, we study the two combinatorial properties on $\Gamma$  that are needed in order to apply  Theorem \ref{thm:nbpossols}: being (part of) a regular triangulation and having (many, hopefully all) positively decorated simplices. As we will see, these properties turn out to be dual to one another. Our combinatorial framework is that of pure, abstract simplicial complexes:

\begin{definition}
  A \emph{pure abstract simplicial complex} of dimension $d$ on $n$ vertices
  (abbreviated $(n,d)$-complex) is a finite set $\Gamma=\{\tau_1,\ldots,\tau_\ell\}$,
  where for any $i\in\{1,\ldots,\ell\}$, $\tau_i$ is a subset of cardinality $d+1$ of
  $[n]:=\{1,\dots,n\}$. The elements of $\Gamma$ are called \emph{facets} and their number 
  (the number $\ell$ in our notation) is the \emph{size} of $\Gamma$.
  A subset of cardinality $2$ of a facet is called an edge of $\Gamma$.
\end{definition}

Let $\mathcal A=\{w_1,\ldots,w_n\}$ be a configuration of $n$ points in $\R^d$
(by which we mean an ordered set; that is, we implicitly have a bijection
between $\mathcal A$ and $[n]$). An $(n,d)$-complex $\Gamma$ is said to be
\emph{supported on $\mathcal A$} if the simplices with vertices in $\mathcal A$
indicated by $\Gamma$, together with all their faces, form a geometric
simplicial complex, see \cite[Definition 2.3.5]{Maund}. 
Typical examples of
$(n,d)$-complexes supported on point configurations are the boundary complexes
of simplicial $(d+1)$-polytopes, or triangulations of point sets. The following
definition and proposition relate these two notions:

\begin{definition}\label{def:regular}
  An $(n,d)$-complex $\Gamma$ is said to be
  \emph{regular} if it is isomorphic to a (perhaps non-proper) subcomplex of a regular triangulation of some point
  configuration $\mathcal  A\subset\R^d$.
\end{definition}

\begin{proposition}\label{prop:regular}
  For a pure abstract simplicial complex $\Gamma$ of dimension $d$ the following properties are equivalent: 
  (1) $\Gamma$ is regular.
  (2) $\Gamma$ is (isomorphic to) a proper subcomplex of the boundary complex of a simplicial $(d+1)$-polytope $P$.
\end{proposition}

\begin{proof}
  This is a well-known fact, the proof of which appears e.g. in \cite[Section
  2.3]{GrunShep}. The main tool to show the backwards statement (which is the
  harder direction) is as follows: let $F$ be a facet of $P$ that does not belong to $\Gamma$ and
  let $o$ be a point outside $P$ but very close to the relative interior of $F$. Project $\Gamma$ towards
  $o$ into $F$ to obtain (part of) a regular triangulation of a $d$-dimensional configuration in the hyperplane containing $F$.  
  (This construction is usually called a Schlegel diagram of $P$ in $F$).
\end{proof}

For $\tau\in\Gamma$ a facet and $C$ a coefficient matrix
associated to the point configuration $\mathcal A$, we let $C_\tau$
denote the $d\times (d+1)$ submatrix of $C$ whose columns correspond to
the $d+1$ vertices in $\tau$.

\begin{definition}
  An $(n,d)$-complex $\Gamma$ is \emph{positively decorable}
if there is a $d\times n$ matrix $C$ that positively decorates every
facet of $\Gamma$. That is,
such that every submatrix $C_{\tau}$ corresponding to a facet
$\tau \in \Gamma$ is positively spanning.
\end{definition}

In this language, Theorem~\ref{thm:nbpossols} says that if  there is a regular and positively decorable $(n,d)$-complex of size $\ell$, then $\Xi_{d,n-d-1}\ge \ell$.

We now introduce a notion of complementarity for pure complexes. This notion is
closely related to matroid duality and, in fact, our result that regularity and
positive decorability are exchanged by complementarity is an expression of that
duality, via its geometric (and oriented) version: Gale duality.

\begin{definition}
  Let $\Gamma$ be an $(n,d)$-complex with facets $\{{\tau_1},\ldots,{\tau_{\lvert\Gamma\rvert}}\}$. We call \emph{complement complex of $\Gamma$} and denote
  $\overline{\Gamma}$ the $(n, n-d-2)$-complex with
  facets
  $\{\overline{\tau_1},\ldots,\overline{\tau_{\lvert\Gamma\rvert}}\}$,
  where $\overline{\tau_i} := [n]\setminus\tau_i$.
\end{definition}

\begin{lemma}\label{lem:galeposdecreg}
  An $(n,d)$-complex
$\Gamma$ is positively decorable
 if and only if its complement
$\overline{\Gamma}$ is
a subcomplex of the boundary complex of an $(n-d-1)$-polytope.
\end{lemma}

\begin{proof}
Recall that an (ordered) set of points $\mathcal A =\{w_1,\ldots,w_n\}\in \R^{n-d-1}$ and an (ordered) set of vectors $\{b_1,\ldots,b_n\} \subset \R^{d}$ are \emph{Gale transforms} of one another if the following $(n-d)\times n$ and $d\times n$ matrices have orthogonally complementary row-spaces (that is, if the kernel of one equals the row-space of the other):
\[
\widetilde{A}=
  \begin{pmatrix}
  1&\dots&1\\
  w_1&\dots&w_n
  \end{pmatrix},
\qquad
{C}=
  \begin{pmatrix}
  b_1&\dots&b_n
  \end{pmatrix}
  .
\]
Every set of points affinely spanning $\R^{n-d-1}$ has a Gale transform (construct $C$ by using as rows a basis for the kernel of $\widetilde{A}$) and every set of vectors with $\sum b_i=0$ has a Gale transform (extend the vector $(1,1,\dots,1)$ to a basis of the kernel of $C$).

By construction, Gale transforms have the property that a vector $\lambda\in \R^n$ is the vector of coefficients of a linear dependence of $\{b_1,\dots,b_n\}$ (kernel of $C$) if, and only if, it is the vector of values of some affine functional on $\{w_1,\dots,w_n\}$ (row space of $\widetilde A$). 
This implies the following (see, e.g.,~\cite[Theorem 1, p. 88]{Grun}): 
a subset $\tau\subset [n]$ of size $d+1$ indexes a positively spanning submatrix of $C$ if, and only if, the polytope $\conv(w_1,\dots,w_n)$ has a facet containing exactly the points $\{w_i: i\not\in\tau\}$ (by a dimensionality argument, these points must then be vertices and the facet be a $d$-simplex). 
Indeed, both things are equivalent to the existence of a unique (modulo scalar) nonnegative $\lambda$ with support equal to $\tau$ in the kernel of $C$ and row-space of $\widetilde A$.

Let now $\Gamma$ be an $(n,d)$-complex. 
If $\overline{\Gamma}$ is realized as a subcomplex of the boundary complex of an $(n-d-1)$-polytope $P \subset \R^{n-d-1}$, there is no loss of generality in assuming $P$ to be the convex hull of the vertices of $\overline{\Gamma}$.
Let $w_1,\dots,w_n$ be those vertices (together with $w_i$ arbitrarily chosen in the interior of $P$ if $i\in [n]$ happens not to
be used as a vertex in $\overline{\Gamma}$). Any matrix $C$ constructed as above positively decorates $\Gamma$.

Conversely, if a matrix $C$ positively decorates $\Gamma$ then there is a nonnegative vector $u_\tau$ in the kernel of $C$ with support $\tau$ for each facet $\tau$ of $\Gamma$. 
Thus, $\lambda = \sum_{\tau \in \Gamma} u_{\tau}$ is also in the kernel of $C$ and all its entries are strictly positive. 
Rescaling the columns of $C$ by the entries of $\lambda$ we get a new matrix that has $(1,\dots,1)$ in the kernel and still positively decorates $\Gamma$. 
The Gale transform described above can then be applied, and results in a set of points whose convex hull $P$ contains $\overline{\Gamma}$ as a subcomplex.
\end{proof}

\begin{corollary}
\label{coro:galeposdecreg}
Let $\Gamma$ be an $(n,d)$-complex. Then:
\begin{enumerate}
\item if $\overline{\Gamma}$ is regular then $\Gamma$ is positively decorable.
\item if $\Gamma$ is positively decorable then $\overline{\Gamma}$ is either regular or the boundary complex of a simplicial polytope. If the latter happens then  $\overline{\Gamma}$ minus a facet is regular.
\end{enumerate}
\end{corollary}

The following examples illustrate the need to perhaps remove a facet in part (2) of the corollary. A regular and positively decorable complex may not have a regular complement:

\begin{example}
\label{example:cross-polytope}
Let $\Gamma$ be the $(2^d,d)$-complex formed by the boundary of a cross-polytope of dimension $d+1$.
Observe that $\Gamma=\overline{\Gamma}$ which, by Lemma~\ref{lem:galeposdecreg}, implies 
$\Gamma$ is positively decorable. Yet, $\overline{\Gamma}$ is not regular since it is not a
\emph{proper} subcomplex of the boundary of a $(d+1)$-polytope.
\end{example}

\begin{example}
\label{example:6cyclic}
The complex $\Gamma=\{ \{1,2,3,4\}, \{2,3,4,5\}, \{3,4,5,6\},$ $\{1,4,5,6\}, \{1,2,5,6\},$ $\{1,2,3,6\}\}$ is a proper subcomplex of the boundary of a cyclic $4$-polytope. Its complement is a cycle of length six, so  $\Gamma$ is regular and positively decorable but $\overline \Gamma$ is not regular. 
\end{example}

The following is the main consequence of Lemma~\ref{lem:galeposdecreg}.

\begin{theorem}\label{thm:boundXi}
  Let $\Gamma$  be an $(n,d)$-complex.
\begin{enumerate}
\item If $\Gamma$ is regular and positively decorable, then $\Xi_{d,n-d-1}\geq \lvert\Gamma\rvert$  and $\Xi_{n-d-2,
  d+1}\geq \lvert\Gamma\rvert-1$.
\item If both $\Gamma$ and $\overline{\Gamma}$ are
  regular, then $\Xi_{d,n-d-1}\geq \lvert\Gamma\rvert$ and $\Xi_{n-d-2,
  d+1}\geq \lvert\Gamma\rvert$.
  \end{enumerate}
\end{theorem}

\begin{proof}
The fact that if $\Gamma$ is regular and positively decorable then $\Xi_{d,n-d-1}\geq \lvert\Gamma\rvert$ is merely a rephrasing of Theorem~\ref{thm:nbpossols}. 
The rest  follows  from Lemma~\ref{lem:galeposdecreg}.
\end{proof}

\begin{example}
\label{exm:k=1}
The inequality $\Xi_{1,k}\ge k+1$ from Proposition~\ref{prop:superadditive} is a special case of Theorem~\ref{thm:boundXi}, 
since a path with $k+1$ edges is regular and positively decorable (the decorating matrix alternates $1$'s and $-1$'s).
\end{example}

\section{Relation to bipartite and balanced complexes}
\label{sec:bipartite}

In this section we relate regularity and positive decorability to the following
two familiar notions for pure simplicial complexes:

\begin{definition}
The \emph{adjacency graph} of a pure simplicial complex $\Gamma$ of dimension
$d$ is the graph 
whose vertices are the facets of $\Gamma$, with two
facets 
adjacent if they share $d$ vertices. We say $\Gamma$ is \emph{bipartite}
if
its adjacency graph is bipartite.
\end{definition}

\begin{definition}\cite[Section III.4]{Sta3} A $(d+1)$-coloring of an  $(n,d)$-complex $\Gamma$
  is a map $\gamma:[n]\rightarrow [d+1]$ 
  such that $\gamma(w_1)\neq \gamma(w_2)$ for every
  edge $\{w_1,w_2\}$ of $\Gamma$. 
 If such a coloring exists, $\Gamma$ is called \emph{balanced}.  
\end{definition}

Observe that two complement complexes $\Gamma$ and $\overline{\Gamma}$ have the
same adjacency graph. Thus, if one is bipartite, then so is the other. The same
is not true for balancedness: A cycle of length six is balanced but its complement (the complex $\Gamma$ of
Example~\ref{example:6cyclic}) is not. For instance, the simplices
$\{1,2,3,4\}$ and $\{2,3,4,5\}$ are adjacent, which implies that $1$ and $5$ should get the same color. But this does not work since $\{1,5\}$ is an edge.

Colorings are sometimes called
  \emph{foldings} since they can be extended to a map from $\Gamma$ to the $d$-dimensional
  standard simplex which is linear and bijective on each facet of $\Gamma$.
Similarly, balanced triangulations are sometimes called \emph{foldable}
triangulations, see e.g. \cite{JZ}.

It is easy to show that \emph{orientable} balanced complexes are bipartite. 
(For non-orientable ones the same is not true, as shown by the $(9,2)$-complex $\{123,$ $234,$ $345,$ $456,$ $567$, $678,$ $789,189,129\}$).
We here show that being positively decorable is an intermediate property.
 
Recall that an \emph{orientation} of an abstract $d$-simplex 
$\tau=\{w_{1},\ldots,w_{d+1}\}$ is a choice of
calling ``positive'' one of the two classes, modulo
even permutations, of orderings of its vertices and ``negative'' the other class. 
For example, every embedding $\varphi:\tau\rightarrow \mathbb R^d$ 
of $\tau$ into $d+1$ points not lying in an affine hyperplane induces
a canonical orientation of $\tau$, by calling an ordering $w_{\sigma_1}, \ldots, w_{\sigma_{d+1}}$ positive or negative
according to the sign of the determinant 
\[
\left|\begin{matrix}
\varphi(w_{\sigma_1}) & \ldots & \varphi(w_{\sigma_{d+1})}\\
1&\ldots & 1
\end{matrix}\right|.
\]

If $\tau$ and $\tau'$ are two $d$-simplices with $d$ common vertices, then respective orientations 
of them are called \emph{consistent} (along their common $(d-1)$-face) if replacing in a positive ordering
of $\tau$ the vertex of $\tau\setminus \tau'$ by the vertex of $\tau'\setminus
\tau$ results in a negative ordering of $\tau'$. A pure
simplicial complex is called \emph{orientable} if one can orient all facets in a manner that makes orientations of all neighboring pairs of them consistent.
In particular, every geometric simplicial complex is orientable, since its embedding in $\R^d$ induces consistent orientations.

Observe that if we decorate a (geometric or abstract) $d$-complex $\Gamma$ on $n$ vertices with a $d\times n$ matrix $C$ as we have been doing in the previous sections then each facet inherits a canonical orientation from $C$. When $C$ positively decorates $\Gamma$ these orientations are ``as inconsistent as can be'':

\begin{proposition}\label{prop:simplicialcomploriented}
Let $(\Gamma, C)$ be a positively decorated
pure simplicial complex. Then, the canonical orientations given by $C$ to the facets of \,$\Gamma$ are inconsistent along every common face of two neighboring facets.
In particular, if \,$\Gamma$ is orientable (e.g., if \,$\Gamma$ can be
geometrically embedded in $\R^{\dim(\Gamma)}$) and positively decorable, then its adjacency graph is bipartite.
\end{proposition}

\begin{proof} 
We need to check that the submatrices of $C$ corresponding to two adjacent
facets $\tau$ and $\tau'$, extended with a row of ones, have determinants of the same sign. Without loss of generality assume the matrices (without the row of ones) to be
\[
M_\tau = \left(\begin{matrix}
c_{1} & \ldots & c_d & c_{d+1}\\
\end{matrix}\right)
\quad \text{ and } \quad
M_{\tau'} =\left(\begin{matrix}
c_{1} & \ldots & c_d & c'_{d+1}\\
\end{matrix}\right).
\]
Since $C$ positively decorates $\tau$, and $\tau'$, and since 
${\minor}(M_\tau,d+1) = {\minor}(M_{\tau'},d+1) = 
\left|\begin{matrix}
c_{1} & \ldots & c_d\\
\end{matrix}\right|
$, 
we get that all the signed minors $(-1)^{i}{\minor}(M_\tau,i)$ and $(-1)^{i}{\minor}(M_{\tau'},i)$ have one and the same sign. In particular, the determinants
\[
\left|\begin{matrix}
c_{1} & \ldots & c_d & c_{d+1}\\
1&\ldots & 1 & 1
\end{matrix}\right|
\quad \text{ and } \quad
\left|\begin{matrix}
c_{1} & \ldots & c_d & c'_{d+1}\\
1&\ldots & 1 & 1
\end{matrix}\right|
\]
have the same sign, so the orientations given to $\tau$ and $\tau'$ by $C$ are inconsistent.

The last assertion is obvious: The positive decoration gives us orientations for the facets that alternate along the adjacency graph, while orientability gives us one that is preserved along the adjacency graph. This can only happen if every cycle in the graph has even length, that is, if the graph is bipartite.
\end{proof}

\begin{proposition}\label{prop:orientcoloring} Let $e_i$ be the $i$-th
  canonical basis vector of $\R^d$ and $e_{d+1}=(-1,\ldots, -1)$.  Let $\Gamma$ be a balanced $(n,d)$-complex 
  with $(d+1)$-coloring $\gamma:[n]\rightarrow [d+1]$. Then the matrix $C$ with column vectors $e_{\gamma(1)}, \ldots, e_{\gamma(n)}$ in this order
 positively decorates $\Gamma$.
\end{proposition}

\begin{proof} By construction, every $d\times(d+1)$ submatrix of $C$
  corresponding to a facet of $\Gamma$ is a column permutation of the
  $d\times(d+1)$ matrix with column vectors $e_1,\ldots, e_{d+1}$ in this order.
  This latter matrix is positively spanning, so the statement follows from
  Proposition~\ref{P:invariant}.  \end{proof}

Propositions~\ref{prop:simplicialcomploriented} and~\ref{prop:orientcoloring} imply:

\begin{theorem}
\label{thm:balanced_vs_bipartite}
For orientable pure complexes (in particular, for geometric $d$-complexes in $\R^d$) one has
\[
\text{balanced}\,\Longrightarrow\,
\text{positively decorable}\,\Longrightarrow\,
\text{bipartite}.
\]
\end{theorem}

None of the  reverse implications is true, as the following two examples respectively show.

\begin{example}
\label{exm:noteq}
The $(7,2)$-complex of Figure~\ref{fig:noteq} has a bipartite adjacency
graph but is not balanced. The right hand side of the figure describes a positive
decoration of the simplex. Therefore, positively decorable simplicial complexes are
not necessarily balanced.
\end{example}

\begin{figure}
  \centering
\begin{tikzpicture}[mynode/.style={circle, inner sep=1.5pt, fill=red}, scale =
  1]
  \node[draw,mynode, label=left:{$1$}] (A1) at (0,0) {};
  \node[draw,mynode, label=left:{$2$}] (A2) at (0,2) {};
  \node[draw,mynode, label=above:{$3$}] (A3) at (1,1) {};
  \node[draw,mynode, label=above:{$4$}] (A4) at (2,0) {};
  \node[draw,mynode, label=right:{$5$}] (A5) at (3,2) {};
  \node[draw,mynode, label=right:{$6$}] (A6) at (3,-1) {};
  \node[draw,mynode, label=left:{$7$}] (A7) at (1,-1) {};
  \fill[opacity=0.3, blue] (A1.center) -- (A2.center) -- (A3.center) -- cycle; 
  \fill[opacity=0.3, red] (A2.center) -- (A3.center) -- (A5.center) -- cycle; 
  \fill[opacity=0.3, blue] (A3.center) -- (A4.center) -- (A5.center) -- cycle; 
  \fill[opacity=0.3, red] (A4.center) -- (A6.center) -- (A5.center) -- cycle; 
  \fill[opacity=0.3, blue] (A4.center) -- (A6.center) -- (A7.center) -- cycle; 
  \fill[opacity=0.3, red] (A1.center) -- (A4.center) -- (A7.center) -- cycle; 
  \draw[black] (A1) -- (A2) -- (A5) -- (A6) --
  (A7) -- (A1) -- (A4) -- (A3) -- (A1); 
  \draw[black] (A2) -- (A3) -- (A5);
  \draw[black] (A3) -- (A4);
  \draw[black] (A4) -- (A5);
  \draw[black] (A7) -- (A4) -- (A6);
  \node[draw,mynode] (B0) at (8,.5) {};
  \node (B1) at (7.5,2) {1=6};
  \node (B2) at (9.8,0) {2=7};
  \node (B3) at (6.5,0) {3};
  \node (B4) at (7.5,-1) {4};
  \node (B5) at (9.2,1.7) {5};
  \draw[black,->] (B0) -- (B1);
  \draw[black,->] (B0) -- (B2);
  \draw[black,->] (B0) -- (B3);
  \draw[black,->] (B0) -- (B4);
  \draw[black,->] (B0) -- (B5);
\end{tikzpicture}
\caption{A two-dimensional simplicial complex whose adjacency graph is bipartite (left) and which is positively decorable (right) but not balanced. The white triangle $134$ is not part of the  complex.\label{fig:noteq}}
\end{figure}
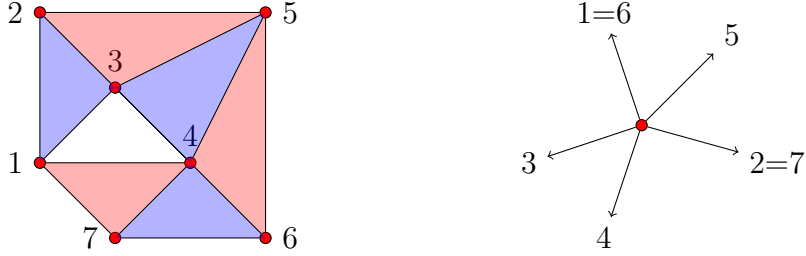

\begin{example}
\label{exm:counterexample}
Let $\Gamma$ be a graph consisting of two disjoint cycles of length four and let $\overline \Gamma$ be its complement, which is a $(8,5)$-complex. 
The adjacency graph of $\Gamma$, hence that of $\overline \Gamma$, is bipartite, again consisting of two cycles of length four. On the other hand, 
since $\Gamma$ is positively decorable but not part of the boundary of a convex polygon, Lemma~\ref{lem:galeposdecreg} tells us that 
$\overline \Gamma$ is regular but not positively decorable (remark that $\overline \Gamma$ cannot be the whole boundary of a simplicial $6$-polytope since for that its adjacency graph would need to have degree six at every vertex).
\end{example}

However, the relationship between balancedness and bipartiteness can be made an equivalence under certain additional hypotheses.
A pure simplicial complex $\Gamma$ is called \emph{locally strongly connected} if 
the adjacency graph of the star of any face is connected.
Locally strongly connected complexes are sometimes called \emph{normal} and they
include, for example, all triangulated manifolds, with or without boundary.
See, e.g., the paragraph after Theorem~A in~\cite{LMS} for more information on them. By results of Joswig \cite[Proposition 6]{J} and \cite[Corollary 11]{J}, a locally
strongly connected and simply connected complex $\Gamma$ on a finite set $\mathcal
A$ is balanced if and only if its adjacency graph is bipartite, see also
\cite[Theorem 5]{Izmes}. In particular, we have:

\begin{corollary}
\label{coro:balanced_vs_bipartite_2}
For simply connected triangulated manifolds (in particular, for triangulations of point configurations) one has
\[
\text{balanced}\,\Longleftrightarrow\,
\text{positively decorable}\,\Longleftrightarrow\,
\text{bipartite}.
\]
\end{corollary}

We close this section by illustrating two concrete applications of Theorem~\ref{thm:balanced_vs_bipartite}.

\begin{corollary} 
\label{coro:manypositive} 
Assume that a finite full-dimensional
  point configuration $\mathcal A$ in $\Z^d$
admits a regular triangulation and let 
 $\Gamma$ be a balanced simplicial subcomplex of this triangulation. Let $\nu: \mathcal A \rightarrow \R$ be a function certifying the regularity of
the triangulation and let $\gamma:\Vertices(\Gamma)\rightarrow [d+1]$ be a
  $(d+1)$-coloring of $\Gamma$.
Then for $t>0$ sufficiently small, the number of positive solutions of the Viro polynomial system
\begin{equation}\label{Virosystem}
\sum_{w \in {\rm Vert}(\Gamma)}
t^{\nu(w)}
e_{\gamma(w)}X^{w}=0
\end{equation}
is not smaller than the number of facets of $\Gamma$. \end{corollary}

\begin{figure}
\centering
\begin{tikzpicture}[mynode/.style={circle, inner sep=1.5pt, fill=red}, scale = 0.7]
    \coordinate (Origin)   at (0,0);
    \coordinate (XAxisMin) at (-3,0);
    \coordinate (XAxisMax) at (5,0);
    \coordinate (YAxisMin) at (0,-4);
    \coordinate (YAxisMax) at (0,4);
    \draw [thin, lightgray,-latex] (XAxisMin) -- (XAxisMax);
    \draw [thin, lightgray,-latex] (YAxisMin) -- (YAxisMax);
  \draw[style=help lines, step= 0.5,dashed,opacity=0.5, lightgray] (-3,-4) grid (5,4);
  \node[draw,mynode, label=above:{$1$}] (A1) at (0.5,-0.5) {};
  \node[draw,mynode, label=left:{$2$}] (A2) at (-2,-3) {};
  \node[draw,mynode, label=above:{$3$}] (A3) at (-2,2) {};
  \node[draw,mynode, label=above:{$4$}] (A4) at (3,0) {};
  \node[draw,mynode, label=above:{$5$}] (A5) at (1.5,3) {};
  \node[draw,mynode, label=above:{$6$}] (A6) at (5,2.5) {};
  \node[draw,mynode, label=right:{$7$}] (A7) at (3,-3) {};
  \fill[opacity=0.3, blue] (A1.center) -- (A2.center) -- (A3.center) -- cycle;
  \fill[opacity=0.3, blue] (A1.center) -- (A4.center) -- (A3.center) -- cycle;
  \fill[opacity=0.3, blue] (A1.center) -- (A2.center) -- (A7.center) -- cycle;
  \fill[opacity=0.3, blue] (A1.center) -- (A4.center) -- (A7.center) -- cycle;
  \fill[opacity=0.3, blue] (A4.center) -- (A5.center) -- (A3.center) -- cycle;
  \fill[opacity=0.3, blue] (A4.center) -- (A5.center) -- (A6.center) -- cycle;
  \draw (A1) -- (A2) -- (A3) -- (A1) -- (A4) --(A3) -- (A5) -- (A4) -- (A6) -- (A5);

  \draw (A4) -- (A7) -- (A1);
  \draw (A2) -- (A7);
\end{tikzpicture}
\caption{The balanced simplicial complex from Example \ref{ex:simcomp6}.
\label{fig:color}}
\end{figure}
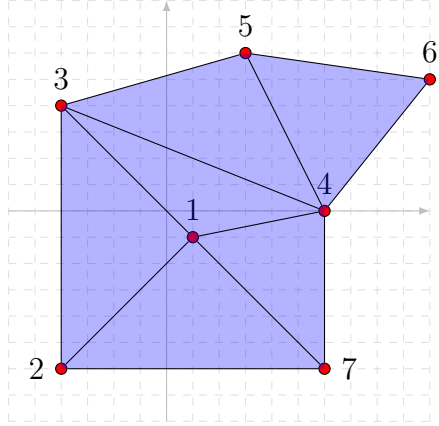

\begin{example}\label{ex:simcomp6}
Let $d=2$, $\mathcal A=\{w_1, \ldots,w_7\}$ where $w_1=(1,-1)$, $w_2=(-4,-6)$, $w_3=(-4,4)$, $w_4=(6,0)$, $w_5=(3,6)$, $w_6=(10,5)$ and $w_7=(6,-6)$,
Choosing heights $\nu(w_1)=\nu(w_2)=\nu(w_3)=0$, $\nu(w_4) = 3$, $\nu(w_5)= 5$, $\nu(w_6) = 10$, and $\nu(w_7) = 2$
provides a regular triangulation of $\mathcal A$ which has the balanced
simplicial subcomplex described in Figure~\ref{fig:color}.
By Corollary~\ref{coro:manypositive}, the Viro polynomial system
\[
\begin{array}{rcl}
X_1X_2^{-1}- X_1^{-4} X_2^4 + t^5X_1^{3}X_2^6 - t^{10}X_1^{10} X_2^{5} - t^2 X_1^6X_2^{-6}=0\\
X_1^{-4} X_2^{-6} - X_1^{-4} X_2^4 + t^3X_1^6 - t^{10}X_1^{10} X_2^{5} - t^2 X_1^6X_2^{-6}=0
\end{array}
\]
has at least six solutions in the positive orthant for $t>0$ sufficiently
small. 
\end{example}

In particular we recover the following result contained implicitly in \cite[Lemma 3.9]{SS}
concerning maximally positive systems.

We use the notation $\Vol(\,\cdot \, )$ for the \emph{normalized volume}, that is, $d!$ times the Euclidean volume in $\R^d$. A triangulation $\Gamma$ of $\mathcal A$ is called \emph{unimodular} if
for any facet $\tau \in \Gamma$ we have $\Vol(\Gamma)=1$. A
polynomial system with support $\mathcal A$ is called \emph{maximally positive} if it
has $\Vol(Q)$ non-degenerate positive solutions, where $Q$ is the convex-hull
of $\mathcal A$. By Kushnirenko Theorem~\cite{Kouch}, if a system is maximally positive then
all its solutions in the complex torus $(\C\setminus\{0\})^d$ lie in the positive orthant $(0,\infty)^d$.

\begin{corollary}[\cite{SS}] \label{coro:manypositiveunimodular}
Assume that $\Gamma$ is a regular  unimodular triangulation of a
finite set $\mathcal A \subset \Z^d$.
Assume furthermore that $\Gamma$ is balanced, or equivalently that its adjacency graph is bipartite. Let $\nu: \mathcal A \rightarrow \R$ be a function certifying the regularity of $\Gamma$ and let $\gamma:\Vertices(\Gamma)\rightarrow [d+1]$ be a
 $(d+1)$-coloring of $\Gamma$. Then, for $t>0$ sufficiently small, the Viro polynomial system
\eqref{Virosystem} is maximally positive.  
\end{corollary}

\begin{proof}
By Corollary \ref{coro:manypositive}, the system \eqref{Virosystem} has at least
$\Vol(Q)$ non-degenerate solutions in the positive orthant for $t>0$ small
enough. On the other hand, it has at most $\Vol(Q)$ non-degenerate solutions
with non-zero complex coordinates by Kushnirenko Theorem \cite{Kouch}.
\end{proof}

This result is also a variant of \cite[Corollary 2.4]{St} which, with the same hypotheses except that of
$\Gamma$ being balanced, concludes that the system \eqref{E:Virosystem} is ``maximally real'':
it has $\Vol(Q)$ non-degenerate solutions in $(\R\setminus\{0\})^d$ (and no other solution in $(\C\setminus\{0\})^d$ by Kushnirenko Theorem).

\section{A lower bound based on  cyclic polytopes}
\label{sec:cyclic}

This section is devoted to the construction and analysis of a family of regular and
positively decorable complexes obtained as 
subcomplexes of cyclic polytopes. 

\begin{definition}\label{def:cyclic} 
  Let $d$ and $n > d+1$ be two positive integers and $a_1<a_2<\dots<a_n$ be real numbers. The cyclic
  polytope $C(n,d+1)$ associated to $(a_1,\ldots, a_n)$ is the convex hull in
  $\R^{d+1}$ of the points $(a_i,a_i^2,\ldots, a_i^{d+1}), i=1,\ldots,n$.
\end{definition}

The cyclic polytope $C(n,d+1)$ is a simplicial $(d+1)$-polytope whose combinatorial structure does not depend on the choice of the real numbers $a_1,\ldots, a_n$. In particular, let us denote by $\mathbf C_{n,d}$ the $d$-dimensional abstract simplicial complex on the vertex set $[n]$ that forms the boundary of $C(n,d+1)$. 
One of the reasons why cyclic polytopes are important is that they maximize the
number of faces of every dimension among polytopes with a given dimension and number of vertices.
We are specially interested in the case of $d$ odd, in which case the complex is as follows:

\begin{proposition}[\cite{LRS}]
\label{prop:cyclic}             
If $d$ is odd, the facets in the boundary of the cyclic polytope $C(n,d+1)$
are of the form
\[
\{i_{1} , i_{1}+1 ,i_{2} , i_{2}+1, \cdots , i_{\frac{d+1}{2}} , i_{\frac{d+1}{2}}+1\}
\]
with $1\le i_{1}$, $i_{\frac{d+1}{2}}\le n$ and $i_{j+1}>i_j+1$ for all $j$.
(If $i_{\frac{d+1}{2}}= n$ then $i_1>1$ is required, and vertex $1$  plays the
role of $i_{\frac{d+1}{2}}+1$).
The number of them equals           
\[
\binom{n- (d+1)/2 -1}{(d+1)/2-1}\quad
+\quad
\binom{n-  (d+1)/2}{(d+1)/2}.
\]
\end{proposition}

Unfortunately, not every proper subcomplex of $\mathbf C_{n,d}$  can be positively
decorated (except in trivial cases) since its adjacency graph is not bipartite.

\begin{example}\label{E:O63}
The tetrahedra
$\{1,2,3,4\}$, $\{1, 2 , 4 , 5\}$, and $\{2 , 3 , 4,5\}$ form a 3-cycle in the adjacency graph of $\mathbf C_{6,3}$.
\end{example}

We now introduce the bipartite subcomplexes of $\mathbf C_{n,d}$ that we are interested in. For the time being, 
we assume both $d+1=2k$ and $n=2m$ to be even.
If we represent any facet $\{ i_{1} , i_{1}+1 ,i_{2} , i_{2}+1, \cdots
,i_{k} , i_{k}+1\}$ of $\mathbf C_{2m,2k-1}$ by
the sequence $\{ i_1,\dots, i_{k}\}$ (where the vertex $1$ plays the
  role of $i_k+1$ if $i_k=n$, as happened in Proposition~\ref{prop:cyclic}) we have a bijection between facets of
$\mathbf C_{2m,2k-1}$ and stable sets of size $k$ in a cycle of length $2m$
(recall that a stable set in a graph is a set of vertices no two of
which are adjacent).
Consider the $(2k-1)$-dimensional
subcomplex ${\mathbf S}_{2m,2k-1}$ of $\mathbf C_{2m,2k-1}$ whose
facets are the $(2k-1)$-simplices
$\{ i_1,\dots, i_{k} \}$ such that for all $j\in [k-1]$ either
$i_j$ is odd or $i_{j+1} - i_j > 2$, and such that either $i_1\ne 2$ or $i_k\ne n$.
That is, we are allowed to take two consecutive pairs to build a simplex if both their $i_j$'s are odd, but not if they are even.
The adjacency graph of the subcomplex $\mathbf S_{2m,2k-1}$ is bipartite,
since the parity of $i_1+\dots+i_{k}$ alternates between adjacent
simplices. 

\begin{example}\label{E:O63_cont}
For $n=6$ and $d+1=4$ we have
\[
\begin{array}{rl}\mathbf S_{6,3}=&\{
  \{ 1, 2, 3, 4 \},
  \{ 1, 2, 4, 5 \},
  \{ 1, 2, 5, 6 \},
  \{ 2, 3, 5, 6 \},
  \{ 3, 4, 5, 6 \},
  \{ 1, 3, 4, 6 \}
\}.
\end{array}
\]
The tetrahedra are written so as to show that the adjacency graph is a cycle: each is adjacent with the previous and next ones in the list.
\end{example}

In order to find out and analyze the number of facets in the simplicial complexes $\mathbf S_{2m,2k-1}$ we introduce the following graphs:

\begin{definition}
\label{def:Delannoy}
The \emph{comb graph} on $2m$ vertices is the graph consisting of a path with $m$ vertices together with an edge attached to each vertex in the path. The \emph{corona graph} with $2m$ vertices is the graph consisting of a cycle of length $m$ together with an edge attached to each vertex in the cycle. Figure~\ref{fig:comb_corona}  shows the case $m=6$ of both.
\end{definition}

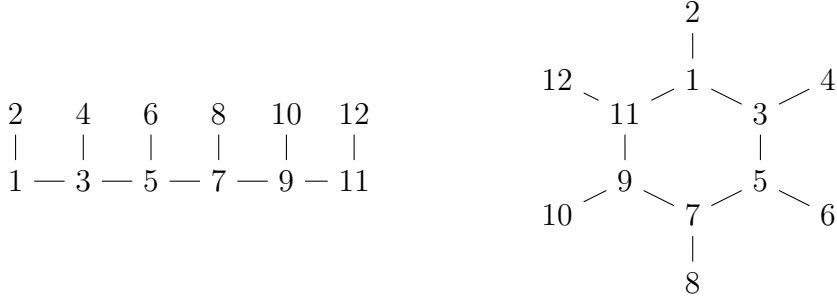
\begin{figure}[htb]
\begin{center}
  \begin{tikzpicture}[scale=1.8]
   \node (1) at (0, 0.25) {$1$};
   \node (3) at (.5,0.25) {$3$};
   \node (5) at (1,0.25) {$5$};
   \node (7) at (1.5,0.25) {$7$};
   \node (9) at (2,0.25) {$9$};
   \node (11) at (2.5, 0.25) {$11$};
   \node (2) at (0, .75) {$2$};
   \node (4) at (.5, .75) {$4$};
   \node (6) at (1., .75) {$6$};
   \node (8) at (1.5, .75) {$8$};
   \node (10) at (2, .75) {$10$};
   \node (12) at (2.5, .75) {$12$};

\draw (1) -- (3) -- (5) -- (7) -- (9) -- (11);
\draw (1) -- (2);
\draw (3) -- (4);
\draw (5) -- (6);
\draw (7) -- (8);
\draw (9) -- (10);
\draw (11) -- (12);
   \node (1) at (5, 1) {$1$};
   \node (3) at (5.5,.75) {$3$};
   \node (5) at (5.5,0.25) {$5$};
   \node (7) at (5,0) {$7$};
   \node (9) at (4.5,0.25) {$9$};
   \node (11) at (4.5, .75) {$11$};
   \node (2) at (5, 1.5) {$2$};
   \node (4) at (6, 1) {$4$};
   \node (6) at (6, 0) {$6$};
   \node (8) at (5, -0.5) {$8$};
   \node (10) at (4, 0) {$10$};
   \node (12) at (4, 1) {$12$};

\draw (1) -- (3) -- (5) -- (7) -- (9) -- (11) -- (1);
\draw (1) -- (2);
\draw (3) -- (4);
\draw (5) -- (6);
\draw (7) -- (8);
\draw (9) -- (10);
\draw (11) -- (12);
  \end{tikzpicture}
  \qquad\qquad
\caption{The comb graph (left) and the corona graph (right) on 12 vertices.}  
\label{fig:comb_corona}
\end{center}
\end{figure}

We denote by $D_{h,k}$  (respectively $F_{h,k}$) the number of matchings of
size $k$ in the comb graph (respectively, the corona graph) with $2(h+k)$
vertices. 
They form sequences A008288 and A102413 in the Online Encyclopedia of Integer Sequences~\cite{OEIS}. The following table shows the first terms:
\[
\begin{array}{c|ccccccc|ccccccc}
&& \multicolumn{5}{c}{D_{h,k}} &&& \multicolumn{5}{c}{F_{h,k}}&
 \\
 \hline
 k=
 & 0 & 1 & 2 & 3 & 4 & 5 & 6 
 & 0 & 1 & 2 & 3 & 4 & 5 & 6 \\
 \hline
h=0: 
&1 & 1 & 1 & 1 & 1 & 1 & 1 
&1 & 1 & 1 & 1 & 1 & 1 & 1 \\
h=1: 
&1 & 3 & 5 & 7 & 9 & 11 & 13 
&1 & 4 & 6 & 8 & 10 & 12 & 14 \\
h=2: 
&1 & 5 & 13 & 25 & 41 & 61 & 85
&1 & 6 & 16 & 30 & 48 & 70 & 96\\
h=3: 
&1 & 7 & 25 & 63 & 129 & 231 & 377
&1 & 8 & 30 & 76 & 154 & 272 & 438\\
h=4: 
&1 & 9 & 41 &129 & 321 & 681 & 1289
&1 & 10 & 48 &154 & 384 & 810 & 1520\\
h=5: 
&1 & 11 & 61 & 231 & 681 & 1683 & 3653
&1 & 12 & 70 & 272 & 810 & 2004 & 4334\\
h=6: 
&1 & 13 & 85 & 377 & 1289 & 3653 & 8989
&1 & 14 & 96 & 438 & 1520 & 4334 & 10672\\
\end{array}
\]

The numbers $D_{h,k}$ are the well-known  \emph{Delannoy numbers}, which have been thoroughly studied~\cite{BS05}. Besides matchings in the comb graph, $D_{h,k}$ equals the number of paths from $(0,0)$ to $(h,k)$ with steps $(1,0)$, $(0,1)$ and $(1,1)$. The equivalence of the two definitions follows from the fact that both satisfy the following recurrence, which can also be taken as a definition of $D_{h,k}$:
\[
D_{h,0}=D_{0,k}=1, \quad \text{and} \quad D_{h,k}= D_{h,k-1} + D_{h-1,k} +
D_{h-1,k-1}, \quad \forall i,j\ge 1. 
\]
The Delannoy numbers can also be defined by either of the formulas in Eq.~\eqref{eq:Delannoy}.

\begin{proposition}
\label{prop:corona}
\[ 
|\mathbf S_{2(h+k),2k-1}|=F_{h,k}=D_{h,k} + D_{h-1,k-1}.
\]
In particular, 
$
D_{h,k} < |\mathbf S_{2(h+k),2k-1}| < 2 D_{h,k}.
$
\end{proposition}

\begin{proof} 
To show that $F_{h,k}=D_{h,k} + D_{h-1,k-1}$, observe that the corona graph is
obtained from the comb graph by adding an edge between the first and last
vertices of the path. We call that edge the reference edge of the corona graph
(the edge $1$---$11$ in Figure~\ref{fig:comb_corona}). Matchings in the corona
graph that do not use the reference edge are the same as matchings in the comb
graph, and are counted by $D_{h,k}$. Matchings of size $i$ using the reference
edge are the same as matchings of size $i-1$ in the comb graph obtained from
the corona by deleting the two end-points of the reference edge; this graph
happens to be  comb graph with $2(h+k-2)$ edges, so these matchings are counted
by $D_{h-1,k-1}$.

To show $|\mathbf S_{2(h+k),2k-1}|=F_{h,k}$, let $m=h+k$. Observe that each simplex in $\mathbf S_{2m,2k-1}$ consists of $k$ pairs
  $(i_j,i_{j}+1)$, $j=1,\dots,k$, with the restriction that when $i_j$ is even then the elements 
  $i_j-1$ and $i_j + 2$ cannot be used. In the
  corona graph, pairs with $i_j$ odd correspond to the spikes and pairs with
  $i_j$ even correspond to the cycle edge between two spikes, which ``uses up''
  the four vertices of two spikes. This correspondence is clearly a bijection.

The last part follows from the previous two since
\[
D_{h,k} < D_{h,k} + D_{h-1,k-1} < 2 D_{h,k}.
\]
\end{proof}

\begin{example}
  Proposition~\ref{prop:corona} says that
\[
|{\mathbf S_{10,5}}|= F_{2,3}= D_{2,3}+D_{1,2} = 25 + 5 = 30.
\]
The following is the whole list of $30$ simplices in $\mathbf S_{10,5}$. Each row is a cyclic orbit, obtained from the first element of the row by even numbers of cyclic shifts. The first two rows, the next three row, and the last row, respectively, correspond to matchings using $0$, $1$ or $2$ edges from the cycle in the pentagonal corona, respectively.
\[
\begin{array}{rl}\mathbf S_{10,5}=&\{
  \{ 1, 2, 3, 4, 5, 6 \},
  \{ 3, 4, 5, 6, 7, 8 \},
  \{ 5, 6, 7, 8, 9, 10 \},
  \{ 1, 2, 7, 8, 9, 10 \},
  \{ 1, 2, 3, 4, 9, 10 \},
\\&
  \{ 1, 2, 3, 4, 7, 8 \},
  \{ 3, 4, 5, 6, 9, 10 \},
  \{ 1, 2, 5, 6, 7, 8 \},
  \{ 3, 4, 7, 8, 9, 10 \},
  \{ 1, 2, 5, 6, 9, 10 \}, 
  \\&
  \{ 1, 2, 3, 4, 6, 7 \},
  \{ 3, 4, 5, 6, 8, 9 \},
  \{ 1, 5, 6, 7, 8, 10\},
  \{ 2, 3, 7, 8, 9, 10 \},
  \{ 1, 2, 4, 5, 9, 10 \},
  \\&
  \{ 1, 2, 3, 4, 8, 9 \},
  \{ 1, 3, 4, 5, 6, 10 \},
  \{ 2, 3, 5, 6, 7, 8 \},  
  \{ 4, 5, 7, 8, 9, 10 \},
  \{ 1, 2, 6, 7, 9, 10 \},
  \\&
  \{ 1, 2, 4, 5, 7, 8 \},
  \{ 3, 4, 6, 7, 9, 10 \},
  \{ 1, 2, 5, 6, 8, 9\},
  \{ 1, 3, 4, 7, 8, 10\},
  \{ 2, 3, 5, 6, 9, 10\},
  \\&
  \{ 1, 2, 4, 5, 8, 9 \},
  \{ 1, 3, 4, 6, 7, 10\},
  \{ 2, 3, 5, 6, 8, 9 \},
  \{ 1, 4, 5, 7, 8, 10\},
  \{ 2, 3, 6, 7, 9, 10 \}
  \}.
\end{array}
\]
\end{example}

The symmetry $F_{h,k}=F_{k,h}$ (apparent in the table, and which follows from
the symmetry in the Delannoy numbers) implies that
$\mathbf S_{2m,2k-1}$ and $\mathbf S_{2m,2m-2k-1}$ have the same size. 
In fact, they turn out to be complementary:

\begin{theorem}
\label{thm:S_decorable}
Let  $\mathbf S'_{2m,2k-1}$ denote the image of $\mathbf S_{2m,2k-1}$ under the following relabelling of vertices: $(1,2,3,4,\dots,2m-1,2m) \mapsto (2,1,4,3,\dots,2m,2m-1)$. (That is, we swap the labels of $i$ and $i+1$ for every odd $i$). Then,
$\mathbf S'_{2m,2k-1}$ is the complement of $\mathbf S_{2m,2m-2k-1}$. 
In particular, $\mathbf S_{2m,2k-1}$ is positively decorable for all $k$ and  regular for $k\ge 2$.
\end{theorem}

\begin{proof} Consider the following obvious involutive bijection $\rho$
  between matchings of size $k$ and matchings of size $m-k$ in the corona
  graph: For a given matching $M$, let $\rho(M)$ have the same edges of the
  cycle as $M$ and the complementary set of (available) spikes. Remember that
  once a matching has been decided to use $i$ edges of the cycle, there are
  $m-2i$ spikes available, of which $M$ uses $k-i$ and $\rho(M)$ uses the other
  $m-k-i$. The relabeling of the vertices makes that, for each odd $i$, if the facet
  of $\mathbf S_{2m,2k-1}$ corresponding to $M$ uses the pair of vertices
  $i+1$ and $i+2$, then in the facet corresponding to $\rho(M)$ we are using the
  complement set from the four-tuple $\{i,i+1,i+2,i+3\}$ (except they have been
  relabeled to $i+1$ and $i+2$ again).

Since the complex $\mathbf S_{2m,2k-1}$ is a subset of the boundary of the cyclic polytope, and a proper subset for $k\ge 2$, it is regular and positively decorable.
\end{proof}

\begin{corollary}
\label{coro:S_decorable}
For every $h,k\in \Z$ with $h>0$, $k>1$, one has
\[
\Xi_{2k,2h} \ge \Xi_{2k-1, 2h}\ge  F_{h,k} \ge D_{h,k}.  
\]
\end{corollary}

\begin{proof}
  The first inequality follows from Remark~\ref{increasing}.
The middle inequality is a direct consequence of Theorem~\ref{thm:S_decorable}
and Theorem~\ref{thm:boundXi}, since $\mathbf S_{2(k+h), 2k-1}$ is regular and positively decorable.
The last inequality follows from Proposition~\ref{prop:corona}.
\end{proof}

\begin{remark}
\label{rem:S_decorable}
The above result is our tightest bound for $\Xi_{d,k}$ when $d$ is odd and $k$ even. For other parities we can proceed as follows:
\begin{itemize}
\item We define $\mathbf S_{2m-1,2k-1}$ to be the deletion of vertex $2m$ in $\mathbf S_{2m,2k-1}$. That is, we remove all facets that use vertex $2m$.
\item We define $\mathbf S_{2m-1,2k-2}$ to be the link of vertex $2m$ in $\mathbf S_{2m,2k-1}$. That is, we keep facets that use vertex $2m$, but remove vertex $2m$ in them.
\end{itemize}

Clearly, $|\mathbf S_{2m,2k-1}| = |\mathbf S_{2m-1,2k-1}| + |\mathbf S_{2m-1,2k-2}|$. Also, since deletion in the complement complex is the complement of the link, we still have that $\mathbf S_{2m-1,2k-1}$ and $\mathbf S_{2m-1,2m-2k-2}$ are complements to one another. 
Moreover, 
since $\mathbf S_{2m,2k-1}$ has a dihedral symmetry acting transitively on vertices and since each facet has a fraction of $k/m$ of the vertices, we have that 
\[
|\mathbf S_{2m-1,2k-1}| = \frac{m-k}{m} |\mathbf S_{2m,2k-1}|
\qquad \text{and}\qquad
|\mathbf S_{2m-1,2k-2}| = \frac{k}{m} |\mathbf S_{2m,2k-1}|.
\]
This, with Corollary~\ref{coro:S_decorable}, implies
\[
\Xi_{2i-1,2j-1} \geq \frac{j}{i+j} \,  F_{i,j} 
\qquad \text{and}\qquad
\Xi_{2i,2j} \geq \frac{i+1}{i+j+1} \,   F_{i+1,j}.
\]

For $\Xi_{2i, 2j-1}$ we can say
\[
\Xi_{2i, 2j-1}\geq
\Xi_{1, 1}\cdot\Xi_{2i-1, 2(j-1)} \ge 2 \, F_{i,j-1}.
\]
\end{remark}

For example, we have that
\begin{equation}
\label{eq:Xi_dd}
\Xi_{d,d} \ge |\mathbf S_{2d+1,d}| = 
\begin{cases}  
 \frac12 |\mathbf S_{2d+2,d}| & \hbox{if $d$ is odd}.\\ 
 \\
 \frac{d/2+1}{d+1} |\mathbf S_{2d+2,d+1}| & \hbox{if $d$ is even}.\\ 
\end{cases}
\end{equation}
The following table shows the lower bounds for $\Xi_{d,d}$ obtained from this
formula, which form sequence A110110 in the Online Encyclopedia of Integer Sequences~\cite{OEIS}:
\[
\begin{array}{cccc}
d& \frac12 |\mathbf S_{2d+2,d}| & \frac{d/2+1}{d+1} |\mathbf S_{2d+2,d+1}| & |\mathbf S_{2d+1,d}| \\
\hline
1  &\frac12 4=&&  2\\
2  &&\frac23 6=&  4\\
3  &\frac12 16=&&  8\\
4&&\frac35 30=&  18\\
5 &\frac12 76=&&  38\\
6 &&\frac47 154=&  88\\
7  &\frac12 384=&&  192\\
8 &&\frac59 810=&  450 \\
9  &\frac12 2004=&&  1002\\
\end{array}
\]

\section{Comparison of our bounds with previous ones}
\label{sec:asymptotics}

In order to derive asymptotic lower bounds on $\Xi_{d,k}$ we
now look at the asymptotics of Delannoy numbers.

\begin{proposition}
\label{prop:delannoy}
For every $i,j\in \Z_{> 0}$ we have
\[
\lim_{n\to \infty} (F_{in,jn})^{1/n} =
\lim_{n\to \infty} (D_{in,jn})^{1/n} = \left( \frac{\sqrt{i^2+j^2}+j}i\right)^i  \left(\frac{\sqrt{i^2+j^2}+i}j\right)^j.
\]
\end{proposition}

\begin{proof}
The first equality follows from Proposition~\ref{prop:corona}. For the second one,
since we have 
\[
D_{i,j}=\sum_{\ell=0}^{\min\{i,j\}} 2^\ell \binom i\ell \binom j\ell,
\]
we conclude that 
\[
\lim_{n\to \infty} (D_{in,jn})^{1/n} = 
\lim_{n\to \infty} \left(2^\ell \binom{in}\ell \binom{jn}\ell\right)^{1/n},
\]
where $\ell = \ell(n) \in [0, \min\{in,jn\}]$ is the integer that maximizes $f(\ell):=2^\ell \binom {in}\ell \binom {jn}\ell$. To find $\ell$ we observe that 
\[
\frac{f(\ell)}{f(\ell-1)}= 
\frac{2(in-\ell)(jn-\ell)}{\ell^2}
=\frac{2(i-\alpha)(j-\alpha)}{\alpha^2},
\]
where $\alpha:=\ell/n$. Since this quotient is a strictly decreasing function of $\alpha$ and since we can think of
 $\alpha \in [0, \min\{i,j\}]$ as a continuous parameter (because we are interested in the limit
 $n\to\infty$), the maximum we are looking for is attained when this quotient equals 1.
This happens when
 \begin{eqnarray}
\alpha^2 = 2 (i-\alpha) (j-\alpha) 
\label{eq:alpha1}
\end{eqnarray}
which implies
 \begin{eqnarray}
\alpha = i + j - \sqrt{i^2+j^2}.
\label{eq:alpha2}
\end{eqnarray}
(We here  take negative sign for the square root since $\alpha = i + j + \sqrt{i^2+j^2} > \min \{i,j\}$ is not a valid solution).
We then just need to plug $\ell=\alpha n$ in $2^\ell \binom{in}\ell \binom{jn}\ell$ and use Stirling's approximation:
\begin{eqnarray*}
 \left(2^{\alpha n} \binom{in}{\alpha n} \binom{jn}{\alpha n}\right)^{1/n} 
&\sim&
 \left(\frac{2^{\alpha n}(in)^{in} (jn)^{jn}}  {(\alpha n)^{\alpha n} ((i-\alpha)n)^{(i-\alpha)n} (\alpha n)^{\alpha n} ((j-\alpha)n)^{(j-\alpha)n} } \right)^{1/n} \\
& = & 
 \frac{2^{\alpha}i^{i} j^{j}}  {\alpha ^{2\alpha } (i-\alpha)^{i-\alpha}  (j-\alpha)^{j-\alpha} } \\
& \overset{(*)}{=} & 
 \frac{2^{\alpha}i^{i} j^{j}}  {(2 (i-\alpha) (j-\alpha) )^\alpha  (i-\alpha)^{i-\alpha}  (j-\alpha)^{j-\alpha} } \\
&=&
 \left( \frac{i}{i-\alpha}\right)^i  \left( \frac{j}{j-\alpha}\right)^j \\
& \overset{(**)}{=}  &
 \left( \frac{i}{\sqrt{i^2+j^2}-j}\right)^i  \left( \frac{j}{\sqrt{i^2+j^2}-i}\right)^j\\
& = &
 \left( \frac{\sqrt{i^2+j^2}+j}i\right)^i  \left(\frac{\sqrt{i^2+j^2}+i}j\right)^j.
\end{eqnarray*}
In equalities (*) and (**) we have used Eqs.~\ref{eq:alpha1} and ~\ref{eq:alpha2}, respectively.
\end{proof}

\begin{theorem}
\label{thm:lower_bound}
For every $d,k\in \Z_{>0}$:
\[
\lim_{n\to\infty} (\Xi_{dn,kn})^{1/n} \ge 
  \left( \frac{{\sqrt{d^2+k^2}+k}}d\right)^{\frac{d}{2}}  \left( \frac{\sqrt{d^2+k^2}+d}k\right)^{\frac{k}{2}}.
\]
\end{theorem}

\begin{proof}
By Corollary~\ref{coro:S_decorable} and Proposition~\ref{prop:delannoy}:
\begin{eqnarray*}
\lim_{n\to\infty} (\Xi_{dn,kn})^{1/n} \ge
\lim_{n\to\infty} \left(D_{\frac{kn}2,\frac{dn}2}\right)^{1/n} \ge
 \left( \frac{{\sqrt{d^2+k^2}+k}}d\right)^{\frac{d}{2}}  \left( \frac{\sqrt{d^2+k^2}+d}k\right)^{\frac{k}{2}}.
\end{eqnarray*}
\end{proof}

Recall from Theorem~\ref{thm:limit_exists} that for every pair $(d,k)$ the limit $\lim_{n\to \infty} {\Xi_{dn,kn}}^{1/(dn+kn)}$ exists and depends only on the ratio $d/k$.
Moreover this limit coincides with the value at $\frac{d}{d+k}$ of the function $\alpha \mapsto \xi_{\alpha,1-\alpha} =
\lim_{n\rightarrow\infty}(\Xi_{\lfloor \alpha n\rfloor,\lfloor (1-\alpha)n\rfloor})^{1/n} \in [1,\infty]$ defined over $(0,1)$ (see Section~\ref{sec:prelim}).
Theorem~\ref{thm:lower_bound} translates to:

\begin{corollary}
\label{coro:lower_bound_xi}
For every $\alpha,\beta > 0$:
\[
\xi_{\alpha,\beta} \ge 
  \left( \frac{{\sqrt{\alpha^2+\beta^2}+\beta}}\alpha\right)^{\frac{\alpha}{2(\alpha+\beta)}}  \left( \frac{\sqrt{\alpha^2+\beta^2}+\alpha}\beta\right)^{\frac{\beta}{2(\alpha+\beta)}}.
\]
\end{corollary}

For example, taking $d=k$ the statement above gives
$$\xi_{1/2,1/2}=   \lim_{d\to\infty}(\Xi_{d,d})^{1/2d} \geq  (\sqrt 2+1)^{1/2}
\approx 1.5538 \ldots.$$

This bound is worse than the one coming from $\Xi_{2,2}\geq 7$ (see
Proposition~\ref{prop:superadditive}), which implies that $(\Xi_{d,
d})^{1/2d}\geq 7^{1/4} \approx 1.6266$. But Corollary~\ref{coro:lower_bound_xi}
gives meaningful (and new) bounds for a large choice of $d/k$ or, equivalently,
of $\alpha\in (0,1)$.  For example, taking $k=2d$ the statement above gives
\[
\xi_{1/3,2/3}=   \lim_{d\to\infty}(\Xi_{d,2d})^{1/3d} \geq   \left(\frac{\sqrt{22+10\sqrt{5}}}2\right)^{1/3} \approx1.4933 \dots.
\]
and the same bound is obtained for $\xi_{2/3,1/3}=   \lim_{d\to\infty}(\Xi_{2d,d})^{1/3d}$.

For better comparison,
Figure~\ref{fig:graph_bound} graphs the lower bound for $\xi_{\alpha,1-\alpha}$ given by Corollary~\ref{coro:lower_bound_xi}.
The red dots are the lower bounds obtained from the previously known values $\Xi_{2,2}\geq 7$ and $\Xi_{d,1} = \Xi_{1,d} = d+1$.
\begin{figure}\centering
\includegraphics[scale=0.4]{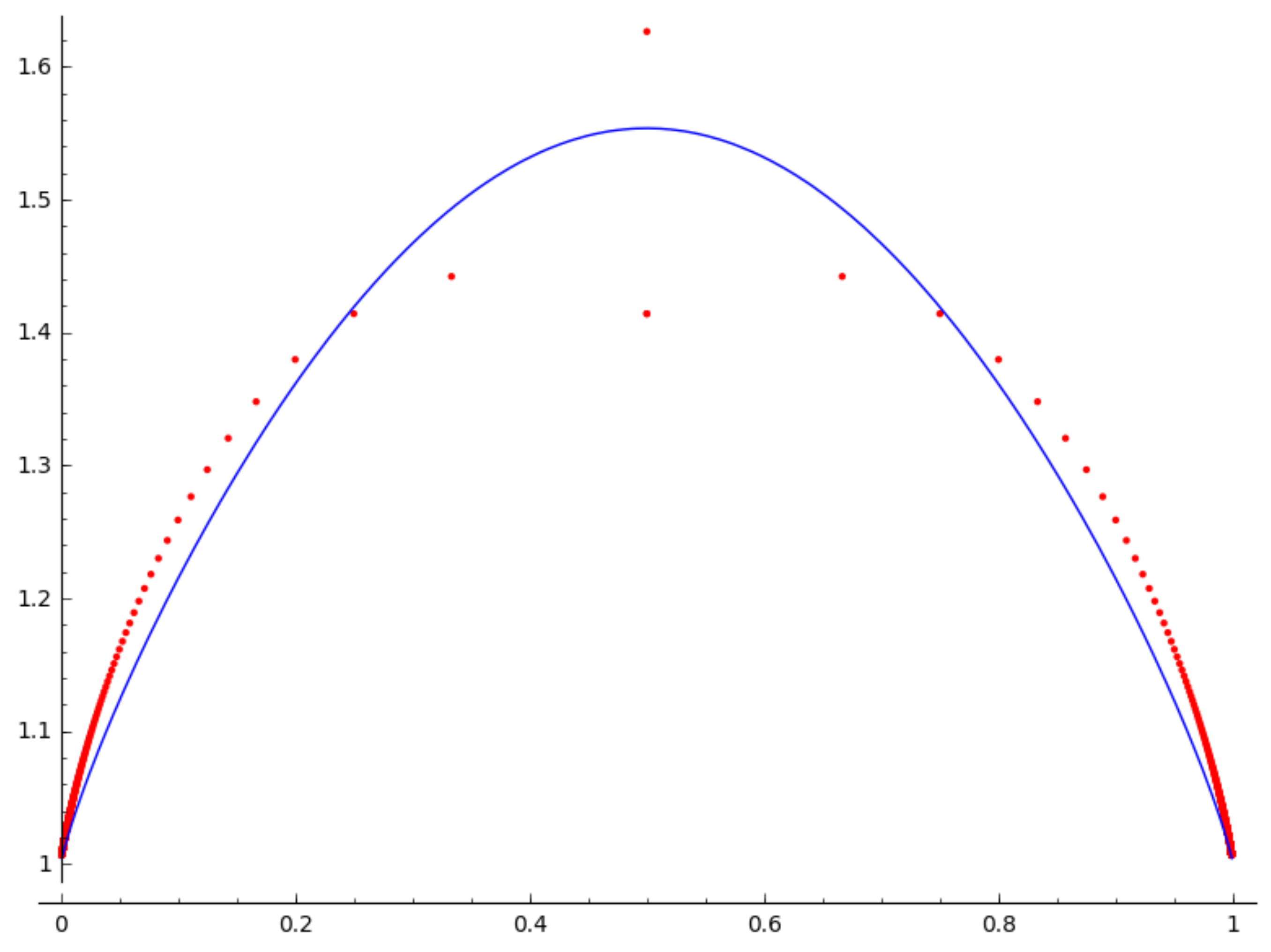}
\caption{The lower bound for $\xi_{\alpha,1-\alpha}$ coming from Theorem~\ref{thm:lower_bound} (blue curve) versus the ones coming from $\Xi_{2,2}\geq 7$ and $\Xi_{d,1} = \Xi_{1,d} = d+1$ (red dots).}
\label{fig:graph_bound}
\end{figure}

Since the function $\alpha\mapsto \xi_{\alpha,1-\alpha}$ is log-concave (Proposition~\ref{prop:concavity}), it is 
a bit more convenient to plot the logarithm of $\xi_{\alpha,1-\alpha}$; in such a plot 
we can take the upper convex envelope of all known lower bounds for $\xi$ and get a new lower bound. 
This is done in Figure~\ref{fig:graph_bound_log} where the black dashed segments show that the use of Corollary~\ref{coro:lower_bound_xi} produces new lower bounds for $\xi_{\alpha,1-\alpha}$ whenever $\alpha\in(0.2,0.8)$.

\begin{figure}\centering
\includegraphics[scale=0.4]{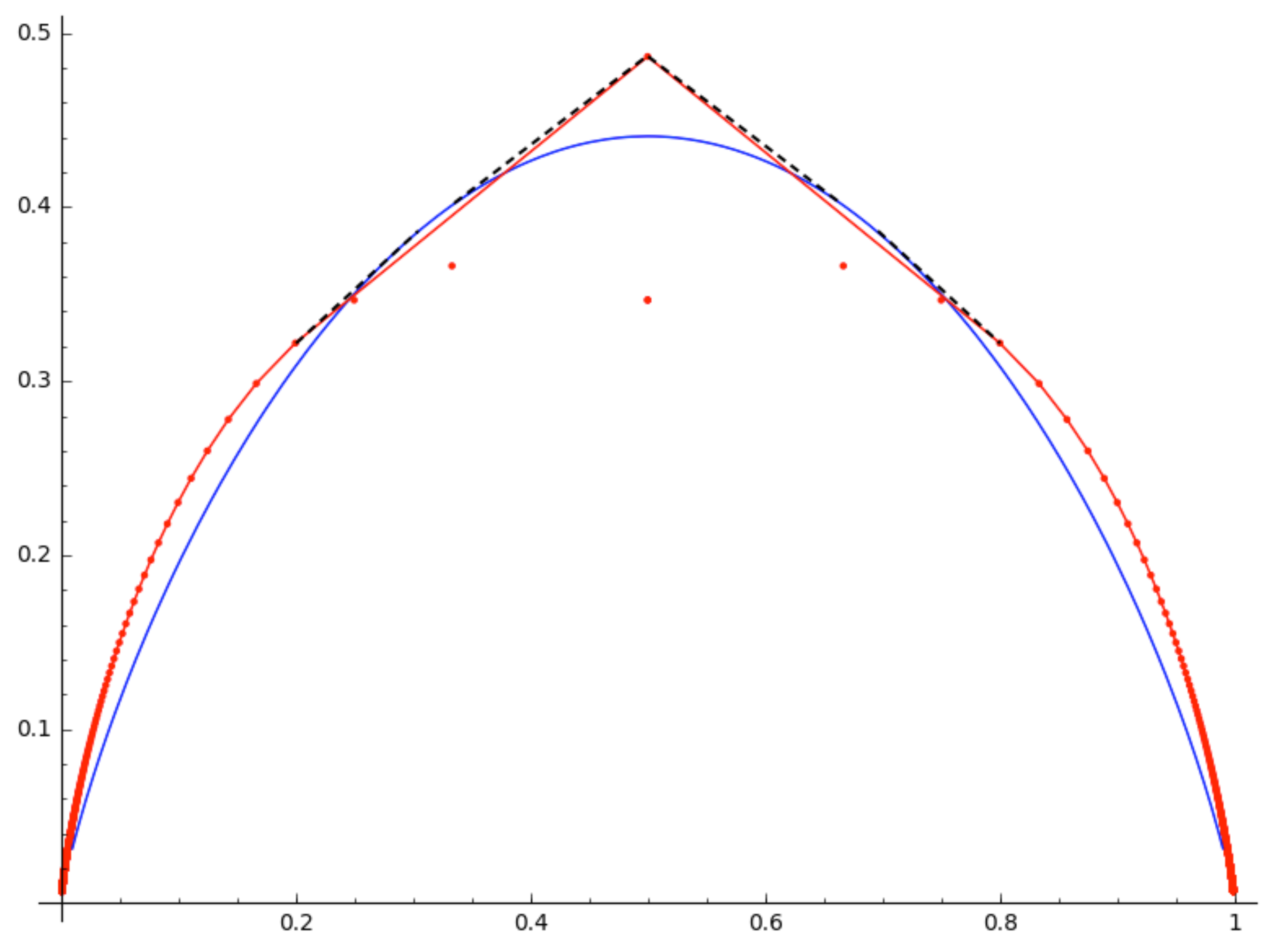}
\caption{The different lower bounds for $\log\xi_{\alpha, 1-\alpha}$, $\alpha\in(0,1)$. The red
line is the best previously known lower bound,
using Proposition~\ref{prop:superadditive} and log-concavity. Our lower
bound (blue curve) is above the previously known ones for $\alpha\in[0.2434,
0.3659]$. This range can be extended to $\alpha\in[0.2, 0.8]$ using 
log-concavity (dashed lines).}
\label{fig:graph_bound_log}
\end{figure}

\begin{remark} 
Our results are stated asymptotically, but one can also compute explicit examples where
the bound of Corollary~\ref{coro:S_decorable} gives systems with more solutions that was previously
achievable. For example, for $d=115$ and $k=264$ the maximum number of roots
obtainable combining the results in Proposition~\ref{prop:superadditive} is
$2.008 \cdot 10^{62}$
while Corollary~\ref{coro:S_decorable} gives $4.073 \cdot 10^{62}$.
\end{remark}

\section{Limitations of the polyhedral method}
\label{sec:Xi_vs_R}

We finish the paper with an analysis of how far could our methods be possibly taken. For this,
let us denote by $R_{d,k}$ the maximum size (i.e. the maximum number of facets) of a regular $(d-1)$-complex on
$d+k$ vertices such that its complement is also regular. Part (2) of Theorem~\ref{thm:boundXi} says 
\[
\Xi_{d,k} \ge R_{d+1,k}
\]
and our main result in Section~\ref{sec:cyclic} was the use of this inequality
to provide new lower bounds for $\Xi_{d,k}$. Observe that either $R_{d,k}$ or
$R_{d,k}-1$ equals the maximum size of a regular positively decorable complex (Corollary~\ref{coro:galeposdecreg}).

\begin{remark}
\label{rem:symmetry}
Our shift on parameters for $R_{d,k}$ is chosen to make it symmetric in $k$ and $d$: $R_{d,k}=R_{k,d}$. 
\end{remark}

The inequality $\Xi_{d,k} \ge R_{d+1,k}$ is certainly not an
equality, as the following table of small values shows:
\[
\begin{array}{ccc}
R_{d+1,k} && \Xi_{d,k}
\\
    \begin{array}{l | cccc}
    {}_{d}\backslash {}^k & 1&2&3&4\\
    \hline
    0&  1&1&1&1\\
    1&  1&3&4&5\\
    2&  1&4&7&8\\
    3&  1&5&8 &\ge 16\\
    \end{array}
&    \qquad\qquad\qquad &
    \begin{array}{l | cccc}
    {}_{d}\backslash {}^k & 1&2&3&4\\
    \hline
    0&  &&&\\
    1&  2&3&4&5\\
    2&  3&\ge7&\\
    3&  4&&\\
    \end{array}
\end{array}
\]
The values of $\Xi$ come from Proposition~\ref{prop:superadditive} and those of $R$ come from:
\begin{itemize}
\item $R_{1,k}=R_{k,1}=1$ is obvious: a regular $0$-dimensional complex can only have one point.
\item $R_{2,k}=R_{k,2} = k+1$ since the largest regular $1$-complex with $2+k$ vertices is a path of $k+1$ edges, and its complement is regular too (Example~\ref{exm:k=1}).
\item $R_{3,k}\le 2k+1$ follows from the fact that a triangulated $2$-ball with
  $k+3$ vertices has at most $2k+1$ triangles (with equality if and only if its
  boundary is a 3-cycle). On the other hand, it is easy to construct a balanced
  $3$-polytope with $k+3$ vertices for every $k\not\in \{1,2,4\}$: for odd $k$,
  consider the bipyramid over a $(k+1)$-gon; for even $k$, glue an octahedron into a facet of the latter. This shows that $R_{3,k}=2k+1$ for all such $k$ (but 
$R_{3,4}=8$ instead of $9$, since no balanced $3$-polytope on $7$ vertices exists; the best we can do is a double pyramid over a path of length four).
\item $R_{4,4}\ge 16$ follows from the complex $\mathbf S_{8,3}$, of size $F_{2,2}=16$.
\end{itemize}

It is easy to prove analogues of Equations~\eqref{E:sec} and~\eqref{E:first} for $R$.
Assume for simplicity that both $d$ and $k$ are even and that $d\le k$. Then, by Proposition~\ref{prop:corona}
\[
R_{d,k} \ge  |\mathbf S_{d+k,d-1}|=F_{d/2,k/2} \ge D_{d/2,k/2} \ge \binom{\frac{d+k}2}{\frac d2},
\]
where the last inequality comes from taking the summand $\ell=0$ in Eq.~\eqref{eq:Delannoy}.
For $k=d$ this recovers  Eq.~\eqref{eq:Xidd} (modulo a sublinear factor) since $\binom{d}{d/2}\in \Theta(2^d /\sqrt{d})$. 
More generally, using Stirling's approximation we get:
\[
  R_{d,k} \ge \binom{\frac{d+k}2}{\frac d2}
  \underset{d,k\rightarrow\infty}{\sim}
  \left(\frac{d+k}k\right)^{k/2} \left(\frac{d+k}d\right)^{d/2}
  \sqrt{\frac{d+k}{\pi kd}}.
\]
For constant $d$ and big $k$ we can approximate $\left(\frac{d+k}k\right)^{k/2} \simeq e^{d/2}$ so that 

\[
R_{d,k} \ge  \frac{e^{d/2}}{\sqrt{\pi d}} \left(\frac{k}d+1\right)^{d/2}.
\]
This, except for the constant factor and for the exponent $d/2$ instead of $d$, is close to Equation~\eqref{E:sec}. Doing the same for constant $k$ and big $d$ gives the analogue of Equation~\eqref{E:first}. 

Similarly, one has
\[
R_{d+d',k+k'}\ge R_{d,k} R_{d',k'}
\]
(the analogue of part (1) in Proposition~\ref{prop:superadditive}) since the join of regular complexes is regular and the complement of a join is the join of the complements.

Regarding upper bounds, since the number of facets of a regular complex cannot exceed that of a cyclic polytope we have that
  \[
R_{2d,2k} \le |\mathbf C_{2d+2k,2d-1}| = \binom{2k+d}{d}+\binom{2k+d-1}{d-1},
\]
so that, by using Stirling's formula, we get
\[
\lim_{n\to\infty} {R_{2dn,2kn}}^{1/(2dn+2kn)} \le \left(\frac{d+2k}{2k}\right)^{\frac k{d+k}} \left(\frac{d+2k}d\right)^{\frac d{2(d+k)}}.
\]

Figure~\ref{fig:lowerbounds} shows this upper bound (green line) together with the lower bounds from Figure~\ref{fig:graph_bound} (blue line and red dots). There are many red dots above the green line, meaning that the upper bound for $R$ is smaller than the lower bound for $\Xi$. 
\begin{figure}\centering
\includegraphics[scale=0.4]{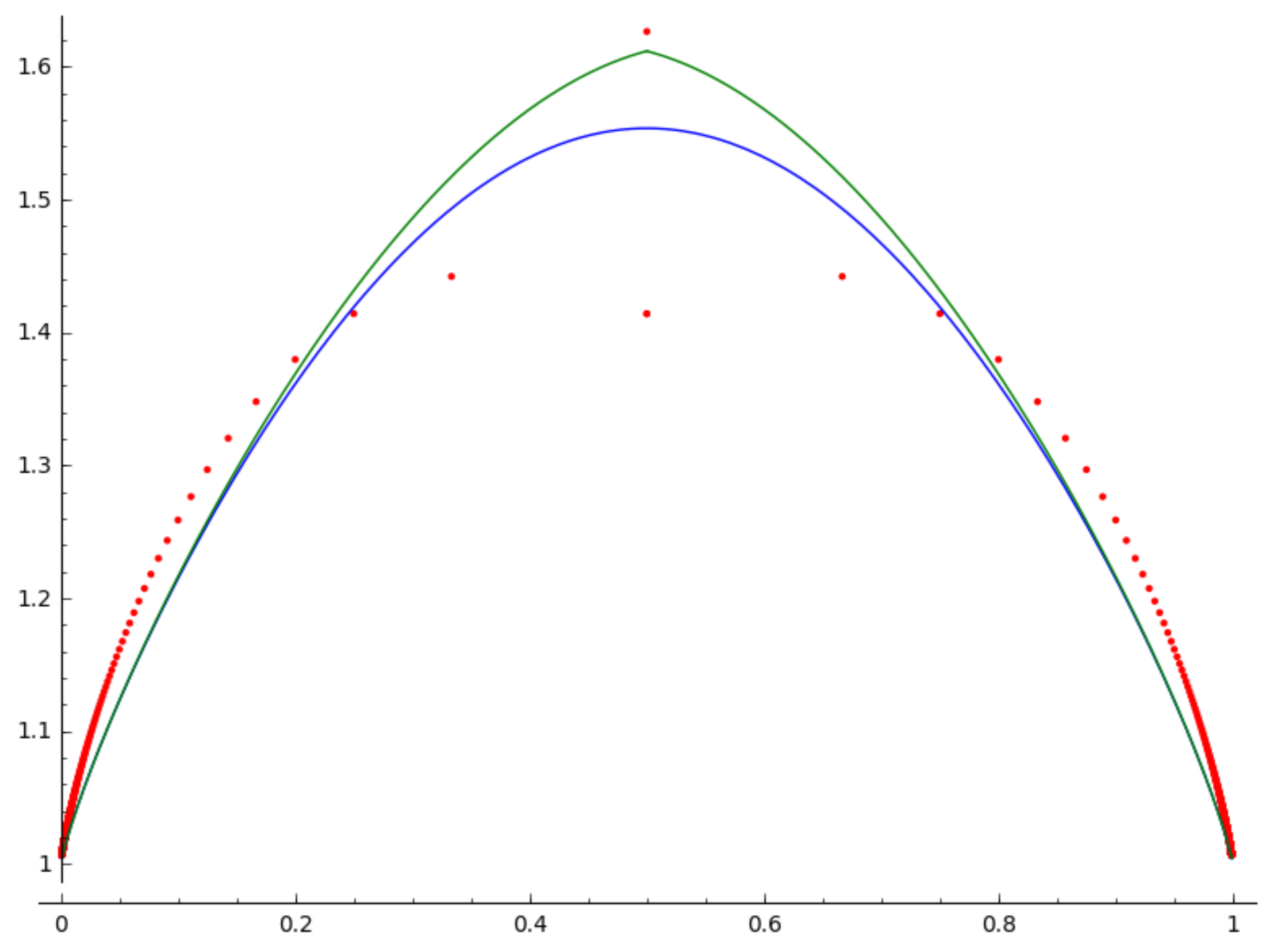}
\caption{The lower bound for $\xi_{\alpha,1-\alpha}$ coming from
Theorem~\ref{thm:lower_bound} (blue curve) versus the ones coming from
$\Xi_{2,2}\geq 7$ and $\Xi_{d,1} = \Xi_{1,d} = d+1$ (red dots). 
The green curve, coming from the upper bound theorem for polytopes,
is the limit of the lower bounds that could possibly be
produced with our method.
}
\label{fig:lowerbounds}
\end{figure}
For example, for the case $d=k$ we have that 
\[
(R_{d,d})^{1/d} \le 3\sqrt{3}/2 \approx 2.598 < 2.6458 \approx 7^{1/2} \le (\Xi_{d,d})^{1/d} .
\]

\providecommand{\bysame}{\leavevmode\hbox to3em{\hrulefill}\thinspace}
\providecommand{\MR}{\relax\ifhmode\unskip\space\fi MR }
\providecommand{\MRhref}[2]{%
  \href{http://www.ams.org/mathscinet-getitem?mr=#1}{#2}
}
\providecommand{\href}[2]{#2}

\bigskip

\footnotesize
\noindent {\bf Authors' addresses:}

\noindent Fr\'ed\'eric Bihan, Laboratoire de Math\'ematiques, Universit\'e Savoie Mont Blanc, Campus Scientifique,
73376 Le Bourget-du-Lac Cedex, France, {\tt frederic.bihan@univ-smb.fr}

\smallskip

\noindent Pierre-Jean Spaenlehauer, CARAMBA project, INRIA Nancy -- Grand Est; Universit\'e de Lorraine; CNRS, UMR 7503; 
LORIA, Nancy, France, {\tt pierre-jean.spaenlehauer@inria.fr}

\smallskip

\noindent Francisco Santos, Depto. de Matem\'aticas, Estad\'istica y Computaci\'on, Universidad de Cantabria; 39012 Santander, Spain, {\tt francisco.santos@unican.es}

\end{document}